\documentclass[reqno,a4paper]{amsart}

\usepackage{general,strings,txfonts,ifthen}
\usepackage[switch*,pagewise]{lineno}

\usepackage[bookmarks, colorlinks=true, pdfstartview=FitV, linkcolor=blue, citecolor=blue, urlcolor=blue]{hyperref}

\newcommand{\ind}[1]{\mathbf{1}_{#1}}
\newcommand{\eps}{\varepsilon}
\newcommand{\inn}[1]{\mathrm{int}\,#1}   %{\left(#1\right)} %interior of a set
\newcommand{\bd}{\mathrm{bd}\,} %boundary of a set
\newcommand{\mydot}{\,\cdot\,}

\newcommand{\asM}{\widetilde{\sM}}
\newcommand{\asC}{\widetilde{\sC}}

\newcommand{\esslim}[1]{\underset{#1}{\rm{esslim}}\,}
\newcommand{\var}{\textnormal{var}}
\newcommand{\cl}[1]{\overline{#1}}
\newcommand{\setmap}{\mathbf{S}}
\newcommand{\Ha}{\mathcal{H}}
\newcommand{\diam}{\textnormal{diam}}
\renewcommand{\R}{\bR}
\newcommand{\N}{\mathbb{N}}
\renewcommand{\conv}{\textnormal{conv}\,}

  \let\oldmarginpar\marginpar
  \renewcommand\marginpar[1]
    {\-\oldmarginpar[\raggedleft\scriptsize #1]%
   {\raggedright \scriptsize #1}}

\numberwithin{equation}{section}
\numberwithin{theorem}{section}
\numberwithin{figure}{section}

\begin{document}
  %\linenumbers

  \title[Minkowski content and fractal curvatures of self-similar tilings]
    {Minkowski content and fractal curvatures of self-similar tilings and generator formulas for self-similar sets}
  %  \title[Minkowski content and fractal curvatures of self-similar sets and tilings]
  %  {Minkowski content and fractal curvatures of self-similar sets and tilings via renewal theory}

  %\author{Michel L. Lapidus}
%  \address{University of California, Department of Mathematics, Riverside, CA 92521-0135 USA}
%  \email{\href{mailto:lapidus@math.ucr.edu}{lapidus@math.ucr.edu}}
%
%  \author{Erin P. J. Pearse}
%  \address{University of Oklahoma, Department of Mathematics, Norman, OK 73019-0315 USA}
%  \email{\href{mailto:ep@ou.edu}{ep@ou.edu}}

  \author{Steffen Winter}
  \address{Karlsruhe Institute of Technology, Department of Mathematics, 76128 Karlsruhe, Germany}
  \email{\href{mailto:steffen.winter@kit.edu}{steffen.winter@kit.edu}}

  \begin{abstract}
  We study Minkowski contents and fractal curvatures of arbitrary self-similar tilings (constructed on a feasible open set of an IFS) and the general relations to the corresponding functionals for self-similar sets. In particular, we characterize the situation, when these functionals coincide. In this case, the Minkowski content and the fractal curvatures of a self-similar set can be expressed completely in terms of the volume function or curvature data, respectively, of the generator of the tiling. In special cases such formulas have been obtained recently using tube formulas and complex dimensions or as a corollary to results on self-conformal sets.
  Our approach based on the classical Renewal Theorem is simpler and works for a much larger class of self-similar sets and tilings. In fact, generator type formulas are obtained for essentially all self-similar sets, when suitable volume functions (and curvature functions, respectively) related to the generator are used.
  We also strengthen known results on the Minkowski measurability of self-similar sets, in particular on the question of non-measurability in the lattice case.

     %We show that under rather general conditions the Minkowski contents of a self-similar set and an associated self-similar tiling coincide and derive in this way some explicit formula for the Minkowski content of self-similar sets in terms of the generator of the tiling.
%     We also introduce and discuss fractal curvatures of self-similar tilings and relate this notion to fractal curvatures of self-similar sets. Also in this case useful formulas for the computation of fractal curvatures of self-similar sets are obtained. The results allow to strengthen known results on the Minkowski measurability of self-similar sets, in particular on the question of non-measurability in the lattice case.
  \end{abstract}

  \date{\textbf{\today}}
  \keywords{self-similar set, self-similar tiling, Minkowski measurability, Minkowski content, generator formula, Renewal theorem}
  \subjclass[2010]{
    Primary: 28A80
    Secondary:  28A12, 28A75
    }
  \thanks{}

\maketitle
\section{Introduction}

Let $A$ be a bounded subset of $\R^d$ and $\eps>0$. Denote by
$$
A_\eps:=\{ z\in\bR^d:\, \inf_{a\in A}|a-z|\leq \eps\}%\quad A_{<r}=\{ z\in\rd:\, d_A(z)<r\}
$$
 its \emph{$\eps$-parallel set} (or $\eps$-parallel neighbourhood), where $|\cdot|$ is the Euclidean norm. Writing $\lambda_d$ for the Lebesgue measure in $\bR^d$, the \emph{$s$-dimensional Minkowski content} of $A$ is the number
 \begin{align}
   \sM^s(A):=\lim_{\eps \searrow 0}\eps^{s-d} \lambda_d(A_\eps),
 \end{align}
provided the limit exists (in $[0,\infty]$). While disfavoured in the past for its theoretical drawbacks compared to other notions in fractal geometry, the Minkowski content has meanwhile proved to be an important notion, in particular in the theory of \emph{complex dimensions} and in connection with the study of spectral properties of domains with fractal boundaries, see \cite{FGCD} and the references therein, but also as a texture parameter (`\emph{lacunarity}') in applications as suggested by Mandelbrot in \cite{Mandelbrot82,Mandelbrot95}. The introduction of fractal curvature measures in \cite{winter08} (see Section~\ref{sec:curv} for defintions) and their systematic study in the last years, e.g.\ in~\cite{zaehle11,WZ,RW09,RW12,RatZa12,KK12,Kom,BZ13,bohl13} has also lead to a deeper understanding of Minkowski contents, in particular of its local properties. Nowadays this notion is viewed in line with the fractal curvatures as one of the geometric characteristics defined in terms of parallel set approximation.

 %More generally, upper and lower Minkowski content $\sM^{*s}(A)$ and $\sM^s_*(A)$ are introduced by replacing the limit by upper and lower limit, respectively. They are always defined.  The numbers
%\begin{align}
 % \overline{\dim}_M A:=\inf\{s>0: \sM^{*s}(A)=0\} \text{ and } \underline{\dim}_M A:=\inf\{s>0: \sM^s_*(A)=0\}
%\end{align}
%are known as the upper and lower \emph{Minkowski dimension} of $A$. If they coincide, the common value is called the \emph{Minkowski dimension} of $A$.
However, computing explicitly the Minkowski content of a given bounded set $A$, remains a challenging task in general (very much like determining the exact Hausdorff or packing measure of a set $A$). Even the question, whether the Minkowski content $\sM^s(A)$ of $A$ exists for a suitable $s$, is not easy to decide in general.
Recall that the set $A$ is \emph{Minkowski measurable}, if $\sM^s(A)$ exists and is positive and finite for some $s\geq 0$.
%In this case, the Minkowski dimension $\dim_M A$ of $A$ equals $s$ and $A$ is also called Minkowski measurable \emph{of dimension} $s$.
%Some classes of sets are known to be Minkowski measurable, e.g. $(d-1)$-rectifiable sets in $\R^d$ with positive and finite Hausdorff measure are $d-1$-Minkowski measurable.
For subsets $A$ of the real line $\R$, characterizations of Minkowski measurability have been given in terms of the asymptotic behaviour of the \emph{fractal string} associated to the closure of $A$ in \cite{LaPo1} and also in terms of the poles %(with real part $D$)
of associated zeta functions, see \cite{FGCD} and the references therein. In higher dimensions some analogous results hold for fractal sprays, cf.~\cite{LaPo2} and \cite{FGCD}. A different characterization is given in \cite{RW12} for arbitrary bounded sets in terms of the surface area of their parallel sets.

 For self-similar sets satisfying the open set condition more explicit results are available: in $\R$ such sets are known to be Minkowski measurable if and only if they are nonlattice, see \cite{Lap:Dundee,Falconer95}, and Lapidus conjectured in \cite{Lap:Dundee} that the same holds for self-similar sets in $\R^d$. This was partially confirmed by Gatzouras \cite{Gatzouras}, who proved that nonlattice sets are Minkowski measurable, leaving the question open, whether lattice sets are always non-Minkowski measurable. Recently, some progress has been made in \cite{LPW2,Kocak1,Kom}, where an affirmative answer to this question is given under additional assumptions. The results
 are based on the construction of suitable \emph{self-similar tilings} (introduced in \cite{SST} and generalized and investigated further in \cite{GeometryOfSST}), which generalize the notion of fractal strings to higher dimensions, see Section~\ref{sec:pre} for more details on self-similar tilings. The non-Minkowski measurability in the lattice case was shown under a number of rather restrictive assumptions, including the existence of a \emph{compatible} self-similar tiling with a \emph{monophase} generator (i.e., one whose inner parallel volume is a polynomial). The derivation consists of the two steps to first compute the Minkowski content of the associated tiling (cf.~\eqref{eqn:def-Mink-T} for the definition), and then to show that it coincides with the Minkowski content of the self-similar set up to a possible correction term (which we will show below to be negligible for strong open sets). In \cite{LPW2, Kocak1} the first step is achieved by employing a suitable tube formula (from ~\cite{TFCD,LPW1}), in \cite{Kom} symbolic dynamics and a symbolic Renewal theorem of Lalley \cite{Lal89} are used, which allows in fact to treat also pluriphase generators.

A side result of these derivations are explicit formulas for the Minkowski content (in case it exists)
%, i.e., in the non-lattice case)
and for the average Minkowski content (which does always exist for self-similar sets), which, apart from the scaling ratios and dimension, involve only the geometric information of the generator of the tiling. This is remarkable in view of the fact, that previously known formulas (see e.g.\cite{Gatzouras} and \cite{winter08}) involve very different geometric data, namely the intersections of parallel sets of smaller copies of the self-similar sets, cf.~also Theorem~\ref{thm:gatz} below. These are usually more difficult to compute. In \cite{Kom}, such generator formulas are obtained also for fractal curvatures and even for some self-conformal sets. Since the focus in this previous work is a different one (non-measurability in \cite{LPW2} and self-conformal sets in \cite{Kom}), the proofs are rather long and technical and it is not easy to see, which of the various assumptions  allow which conclusions and, in particular, which are really needed in the self-similar setting for such generator formulas to hold. This is the central question to be addressed in this paper. More precisely, we study the following problems:
%It is a remarkable finding (and not obvious at all) that the newly derived expressions coincide with the old ones but they do as they both determine the Minkowski content of the same set.

%In this note, we address the following questions:
1) What can be said in general about the Minkowski content (and the fractal curvatures) of self-similar tilings on their own?
2) What is the general relation between the Minkowski contents (and fractal curvatures) of a self-similar set and suitable associated tilings? 3) Under which conditions is it possible to express the Minkowski content (the fractal curvatures) of a self-similar set in terms of the generator of an associated tiling, that is, when does there exist a generator formula for the Minkowski content (the fractal curvatures) of a self-similar set?
We address these questions separately first for Minkowski contents (in Section \ref{sec:main}) and then  for the other fractal curvatures (in Section~\ref{sec:curv}).

Concerning Minkowski contents, the first question is answered for arbitrary self-similar tilings as follows: Using renewal theory, we show that under a very mild and natural assumption on the generator $G$ of the tiling (namely that the dimension of the boundary of $G$ is smaller than the similarity dimension of the underlying IFS), a similar lattice-nonlattice dichotomy holds for the Minkowski content of tilings as for self-similar sets. That is, we prove a counterpart of Gatzouras' Theorem for self-similar tilings, see~ Theorem~\ref{thm:main}. Moreover, we obtain a simple and general generator formula for the (average) Minkowski content of a self-similar tiling in terms of its generator $G$, see Corollary~\ref{cor:main}, which specializes to the known expressions in the previously studied cases of monophase and pluriphase tilings, cf.~Corollaries \ref{cor:monophaseR} and \ref{cor:pluriphase}.
Concerning the second question, we demonstrate that for self-similar tilings constructed on a \emph{strong} feasible open set $O$ compatibility is sufficient for the (average) Minkowski contents of the set and the tiling to coincide. No further assumptions are required. In particular, the generator does not need to be monophase and the contribution of the parallel sets of $O$ (which appears e.g.\ in the formulas in \cite{LPW2}) is always negligible, see Theorem~\ref{thm:Minkowski-measurability-fractals}. We emphasize that our main point here is not the formula itself (which is known in special cases from \cite{LPW2,Kocak1,Kom} and even holds for certain self-conformal sets, see \cite{Kom}), but the generality of its validity and its remarkably simple proof based on the classical Renewal Theorem.

While it is now clear, that for all self-similar sets which possess a compatible tiling, the Minkowski content can be expressed by a generator formula, it is also well known that not all self-similar sets possess compatible tilings, see \cite{GeometryOfSST,PokWi}.  In view of the third question, it is therefore natural to ask whether generator formulas can be obtained also in non-compatible situations.
It turns out that essentially all self-similar sets allow generator formulas, as long as they possess a tiling, that is, as long as they are not full-dimensional, see Theorem~\ref{thm:main-main}.
The key to this is to change our point of view on tilings. We study for a tiling constructed on a strong feasible set $O$ of a self-similar set $F$ the volume function $\lambda_d(F_\eps\cap O)$ instead of the parallel volume of the tiling. We show that the thus modified Minkowski content (which is in fact, the relative Minkowski content of $F$ relative to the set $O$) does always coincide with the ordinary Minkowski content of $F$. Moreover, a generator formula holds for some pair $(F,O)$ if and only if $F$ is not full-dimensional and the set $O$ satisfies a certain projection condition (cf.~\eqref{eqn:PC}, p.\pageref{eqn:PC}), see Theorem~\ref{thm:Minkowski-measurability-fractals2}. But due to an observation of E.\ Pearse, there is always a strong feasible set $O$ satisfying this projection condition, see Proposition~\ref{prop:Vc}.
Hence self-similar tilings may be used to compute Minkowski contents of self-similar sets in general. No compatibility is needed. In particular, the results apply also to self-similar sets in $\R^d$ of dimension less than $d-1$ and e.g.\ to the Koch curve.

The results regarding Minkowski contents allow as well to strengthen the known statements on the non-Minkowski measurability in the lattice case by removing some of the several assumptions made in earlier work on this question, cf.~Corollary~\ref{cor:monophase} and Remark~\ref{rem:monophase-no-comp}.
We hope that our results will also push forward the resolution of Lapidus' conjecture in the general case, as the general generator formula for the Minkowski content may help to find the right tubular zeta function required to extend e.g.\ the proofs in \cite{LPW2}.
%One motivation for this investigation was to get an idea how a geometric zeta function mightto  Moreover, they give a hint which expressions

In Section~\ref{sec:curv}, we first introduce and study fractal curvatures of self-similar tilings. A counterpart of Gatzouras' Theorem for fractal curvatures of self-similar tilings is proved, which parallels results obtained for self-similar sets in \cite{winter08,zaehle11}.
 %and has again a surprisingly simple proof based on the Renewal Theorem.
A formula expressing the fractal curvatures of a tiling in terms of the curvature data of
its generator is obtained. Concerning the second question, namely the relations between the fractal curvatures of self-similar sets and associated tilings, we show that under compatibility, analogous results hold as for the Minkowski contents. The fractal curvatures of self-similar sets are expressed by generator formulas, see Theorem~\ref{thm:comp}. For our results, the usual assumptions required to ensure the existence of fractal curvatures of self-similar sets (regularity of the parallel sets, curvature bound condition) suffice -- when combined with compatibility. More generally, we show that generator formulas hold also in non-compatible situations, provided the projection condition \eqref{eqn:PC} holds, see Theorem~\ref{thm:no-comp-curv}. Here again the tiling is used to partition $O$ and the curvature measures $C_k(F_\eps,R)$ of $F_\eps$ inside the tiles $R$ are studied rather than the curvature measures of the parallel sets $R_{-\eps}$ of $R$. This way we recover and extend in the self-similar setting the results obtained in \cite{Kom}.

%Remark für später: These conditions can also be reformulated in terms of the generator of the tiling, making it easier to decide whether these conditions are satisfied.

We remark that the classical Renewal Theorem (on which these results as based) turned out to be a perfect tool for studying self-similar tilings as well as the required relative Minkowski contents and relative fractal curvatures. The renewal equations are simpler than the ones occurring for self-similar sets. Essentially, the only other tool used are some estimates derived in \cite{winter08} and \cite{WZ}, respectively.

In Section~\ref{sec:pre}, we recall self-similar sets and self-similar tilings and introduce some notation. Section~\ref{sec:main} is devoted to the results on Minkowski contents, while in Section~\ref{sec:curv}, we study fractal curvatures of self-similar sets and tilings.

\section{Preliminaries} \label{sec:pre}

Let $N\in\bN$, $N\geq 2$ and let $\{S_1,\ldots, S_N\}$ be an iterated function system (IFS) consisting of contracting similarities $S_i:\bR^d\to\bR^d$ with contraction ratios $r_i\in(0,1)$, $i=1,\ldots, N$. It is well known that for each such IFS there is a unique nonempty compact set $F$ satisfying the invariance relation $\setmap F=F$, where $\setmap$ is the set mapping defined by
\begin{align}
\setmap(A)=\bigcup_{i=1}^N S_i(A),\quad A\subset\bR^d,
\end{align}
see \cite{Hut}.  $F$ is called the \emph{self-similar set} generated by the IFS $\{S_1,\ldots, S_N\}$. To avoid strong overlaps of the pieces $S_i F$ in the union set $F$, frequently the following assumption is made on the IFS, called the \emph{open set condition} (OSC): There exists a nonempty and bounded open set $O$ such that for all $i\neq j$
\begin{align} \label{eqn:OSC}
  S_i O\subset O\quad \text{ and } \quad S_i O\cap S_j O= \emptyset.
\end{align}
If additionally $O$ is assumed to satisfy $O\cap F\neq \emptyset$, then this condition is called the \emph{strong open set condition} (SOSC). In the present setting, OSC and SOSC are known to be equivalent, cf.~\cite{Schief}, although not every set $O$ satisfying \eqref{eqn:OSC} does contain a point of $F$. We call any set $O$ satisfying \eqref{eqn:OSC} a \emph{feasible open set} for $F$ (or the IFS $\{S_1,\ldots, S_N\}$) and if $O$ satisfies additionally $O\cap F\neq\emptyset$, we call it a \emph{strong feasible open set} for $F$.

Let $D\in\bR$ be the unique real solution $s$ of the equation
\begin{align*}
  \sum_{i=1}^N r_i^s=1.
\end{align*}
$D$ is called the \emph{similarity dimension} of $F$ (or of the IFS $\{S_1,\ldots, S_N\}$). It is well known that under OSC $D$ coincides with the Minkowski dimension (and other dimensions) of the set $F$.
In \cite{GeometryOfSST}, for any fixed feasible open set $O$, a tiling $\sT=\sT(O)$ of $O$ has been associated to the IFS $\{S_1,\ldots, S_N\}$, which is defined as follows: Denote by $\Sigma_N^n:=\{1,\ldots, N\}^n$ the family of all words of length $n\in\bN_0$ formed by the alphabet $\{1,\ldots, N\}$ and let $\Sigma_N^*:=\bigcup_{n=0}^\infty\Sigma_N^n$ be the family of all finite words. For $\sigma=\sigma_1\ldots\sigma_k\in\Sigma_N^*$, we use the abbreviations $S_\sigma:=S_{\sigma_1}\circ S_{\sigma_2}\ldots\circ S_{\sigma_k}$ and $r_\sigma:=r_{\sigma_1}\cdot r_{\sigma_2}\cdot\ldots\cdot r_{\sigma_k}$.

We write $K:=\cl{O}$ for the closure of the set $O$ and set $G:=O\setminus \setmap K$. Observe that $G$ is open. The tiling $\sT(O)$ is the set family
\begin{align}
  \sT(O):=\{S_\sigma G : \sigma\in\Sigma_N^*\}
\end{align}
of the iterates of $G$ under the mappings of the IFS. $\sT(O)$ is a tiling of the set $O$ in the sense that the elements of $\sT(O)$ are pairwise disjoint and that the closure of their union coincides with the closure $K$ of $O$, that is, we have the decomposition
$$
K=\cl{\bigcup_{R\in\sT} R},
$$
see \cite[Thm.~5.7]{GeometryOfSST}.
We call the set $G$ the \emph{generator} of $\sT$ and we write $T:=\bigcup_{R\in\sT} R$ for the union set of all \emph{tiles} of $\sT$. Observe that $T$ is open, since all tiles are open.

\section{Minkowski content of self-similar sets and tilings} \label{sec:main}

We start by recalling a result on the existence of the Minkowski content for self-similar sets, which is essentially due to Gatzouras \cite[Theorems 2.3 and 2.4]{Gatzouras} (except for the case $d=1$, which was obtained earlier in \cite{Lap:Dundee} and \cite{Falconer}) and which can be proved using some Renewal Theorem. Recall that for a compact set $A\subset\bR^d$ and $s\geq 0$, the ($s$-dimensional) \emph{average Minkowski content} $\asM^s(A)$ is defined by
\begin{align} \label{eq:avMC}
   \asM^s(A):=\lim_{\delta \searrow 0}\frac{1}{|\ln\delta|} \int_\delta^1
\eps^{s-d} \lambda_d(A_\eps) \frac{d\eps}{\eps},
\end{align}
whenever this limit exists.

\begin{theorem} {\upshape(Gatzouras' Theorem)}\label{thm:gatz}\\
Let $F\subset\bR^d$ be a self-similar set satisfying OSC and let $D$ be the similarity dimension of $F$.
Then the average Minkowski content $\asM^D(F)$ of $F$ exists and coincides with the strictly positive value
\begin{equation} \label{eqn:Xd}
X_d:=\frac{1}{\eta} \int_0^1 \eps^{D-d-1} R_d(\eps)\ d\eps,
\end{equation}
where the function $R_d:(0,\infty)\to \bR$ is given by
\begin{equation} \label{Rd-def}
R_d(\eps) = \lambda_d(F_\eps)-\sum_{i=1}^N \ind{(0,r_i]}(\eps) \lambda_d((S_i F)_\eps)\,
\end{equation}
and $\eta= - \sum_{i=1}^N r_i^D \ln r_i$.
If $F$ is nonlattice, then also the Minkowski content $\sM^D(F)$ of $F$ exists and equals $X_d$.
\end{theorem}

\paragraph{\bf Minkowski content of self-similar tilings.} Our first aim is to provide an analogous result for arbitrary self-similar tilings. Similar to Gatzouras' Theorem, it is derived by employing the Renewal Theorem. However, its derivation is surprisingly simple and general. %move: The main motivation is to come up with a general expression for the Minkowski content of a tiling (which may help to make a qualified guess for the geometric zeta functions for such tilings). %
In the special case of a monophase generator, we will recover from this general statement the expressions for the Minkowski content derived in \cite{LPW1,LPW2} by means of fractal tube formulas and complex dimensions.

We recall the definition of the (inner) Minkowski content of a self-similar tiling.
%bounded open set $A\subset\bRd$, in which the full parallel sets are replaced by the inner parallel sets, that is, by the parallel sets of $\bd A$ relative to the interior of $A$.
%For $\ge>0$, let
%\begin{align}
%  A_{-\ge}= \{x \in A : \dist(\bd A, x) \leq \ge\}
%  \end{align}
%denote the \emph{(inner) $\ge$-parallel set} of $A$.
%We write $V(A ,\ge) := \gl_d\left(A_{-\ge}\right)$ for (inner) $\ge$-parallel volume of $A$.
In \cite{HeLa2, Zubrinic} the \emph{($s$-dimensional) relative Minkowski content} of a bounded set $A\subset\R^d$ \emph{relative to a set} $\Omega\subset\R^d$ is defined as the number
 \begin{align} \label{eqn:inner_Mink_def}
   \sM^s(A,\Omega):=\lim_{\eps \searrow 0}\eps^{s-d} \lambda_d \left( A_{\eps}\cap \Omega \right),
 \end{align}
provided the limit exists (in $[0,\infty]$). The Minkowski dimension $\dim_M(A,\Omega)$ and the average Minkowski content $\widetilde\sM^s(A,\Omega)$ of $A$ \emph{relative to} $\Omega$ are analogously defined in the obvious way. The \emph{inner Minkowski content} of a bounded open set $U\subset\R^d$ is the relative Minkowski content $\sM^s(\bd U, U)$ of $\bd U$ relative to $U$. Observe that $\sM^s(\bd U, U)$ is equivalently given by the limit $\lim_{\eps \searrow 0}\eps^{s-d}V(U,\eps)$, where
\begin{align*}
   V(U,\eps):=\lambda_d(U_{-\eps})
\end{align*}
is the volume of the \emph{inner $\eps$-parallel set}
\begin{align} \label{eq:inner-par-set}
U_{-\ge}:= \{x \in U : d(\bd U, x) \leq \ge\}
\end{align}
of $U$.
%In fact, it would be more appropriate to call the above limit the relative Minkowski content of $\bd A$ relative to $A$, as e.g.\ discussed in. However, for notational simplicity, we use in this note the same symbol for the inner and the usual Minkowski content. It will be clear from the context which one is meant. The inner version is exclusively applied to open sets, while the usual one is applied to compact sets. and again the same symbols are used.
For self-similar tilings $\sT$, it is convenient to write
\begin{align} \label{eqn:def-Mink-T}
\sM^D(\sT)&:=\sM^D(\bd T, T) \quad \text{ and } \quad \asM^D(\sT):= \asM^D(\bd T,T).
\end{align} %and $\dim_M \sT:=\dim_M T$.

The counterpart of Theorem~\ref{thm:gatz} for self-similar tilings $\sT$ reads as follows. As before $D$ is the similarity dimension of the underlying IFS (which coincides with the Minkowski dimension of the generated self-similar set). Recall that the generator of $\sT=\sT(O)$ is the open set $G=O\setminus\setmap \cl{O}$. Let $g$ denote the inradius of $G$.

\begin{theorem} \label{thm:main}
Let $\sT=\sT(O)$ be a self-similar tiling generated on a feasible open set $O$. Assume that the generator $G$ of $\sT$ satisfies $\dim_M(\bd G, G) <D$.
%, where $D$ is the similarity dimension of the associated IFS.
Then the $D$-dimensional average Minkowski content of $\sT$ exists, is strictly positive and given by
\begin{align}\label{eqn:Yd}
\widetilde\sM^D(\sT)=\frac{1}{\eta} \int_0^g \eps^{D-d-1} h(\eps)\ d\eps,
\end{align}
where $\eta= - \sum_{i=1}^N r_i^D \ln r_i$ and the function $h:(0,\infty)\to \bR$ is given by
\begin{equation} \label{h-def}
h(\eps) = V(T,\eps)-\sum_{i=1}^N \ind{(0,r_ig]}(\eps) V(S_i T, \eps)\,.
\end{equation}
If $\sT$ is nonlattice, then also the Minkowski content $\sM^D(\sT)$ of $\sT$ exists and equals $Y_d$.
\end{theorem}

Observe that the hypothesis $\dim_M (\bd G, G)<D$ in Theorem~\ref{thm:main} is equivalent to the following assertion:  There are constants $\gamma, c>0$ such that, for each $0<\eps\le g$,
\begin{equation}\label{eqn:VG}
|V(G,\eps)|\le c \eps^{d-D+\gamma}.
\end{equation}
(Indeed, $\dim_M (\bd G, G)<D$ implies that there is some $\gamma>0$ such that $\dim_M (\bd G,G)<D-\gamma<D$ meaning that $\eps^{D-d-\gamma}V(G,\eps)$ is bounded as $\eps \searrow 0$ and vice versa.) The estimate \eqref{eqn:VG} is exactly the assumption required for the Renewal Theorem. But such an assumption is not only necessary to apply the Renewal Theorem, it is also very natural. It is clear that for self-similar tilings satisfying $\dim_M (\bd G, G)>D$ the conclusion of Theorem~\ref{thm:main} is not true. Since $\bd G\subset \bd T$, the Minkowski dimension of the tiling is at least $\dim_M (\bd G, G)$ in this case and thus strictly greater than $D$.

Note that it is easy to construct self-similar tilings, which do not satisfy the assumption $\dim_M (\bd G, G)<D$, see Example~\ref{ex:Koch} below. Hence this assumption cannot be omitted. However, such examples are kind of artificial. We will see below that the assumption is satisfied for all reasonable tilings, including all that have been studied previously in the literature. In particular, it is satisfied for all compatible tilings as well as for all tilings with monophase or pluriphase generators (independent of compatibility).

\begin{proof}[Proof of Theorem~\ref{thm:main}]
  We use a version of the Renewal Theorem adapted to limits as $\eps\searrow 0$ formulated in \cite[Theorem~4.1.4]{winter08}. Without loss of generality, we can assume that $g=1$. (The general case follows from scaling arguments. Alternatively, Theorem~4.1.4 could easily be reformulated for arbitrary $g>0$, see Remark~\ref{rem:renewal}.)

  We apply it to the functions
  $f(\eps):=V(T,\eps)$ and $\varphi_d(\eps):=h(\eps)$. First, it is easily seen, that due to the definition of $h$, the following renewal equation holds for each $\eps>0$:
  \begin{equation}
h(\eps)=f(\eps)-\sum_{i=1}^N r_i^{d} \ind{(0,r_ig]}(\eps)  f(\eps/r_i).
\end{equation}
Indeed, this equation is transparent from the relation
\begin{align*} %\label{eqn:V-S_iT}
  V(S_i T, \eps)=\lambda_d((S_i T)_{-\eps})=\lambda_d(S_i(T_{-\eps/r_i}))=r_i^d \lambda_d(T_{-\eps/r_i})= r_i^d V(T,\eps/r_i)= r_i^d f(\eps/r_i).
\end{align*}
It remains to show that the hypotheses on $\varphi_d=h$ in \cite[Theorem~4.1.4]{winter08} are satisfied. Since $f$ is continuous in $\eps$, it is obvious that $h$ is piecewise continuous with at most finitely many discontinuities. Moreover, it is easily seen from \eqref{h-def}, that the estimate \eqref{eqn:VG} holds similarly with $h(\eps)$ instead of $V(G,\eps)$. Indeed, for $\eps<\min_i r_i g$, the function $h$ coincides with $V(G,\cdot)$, i.e.\
  \begin{align} \label{eqn:h-G-relation}
    h(\eps)=V(T,\eps)-\sum_{i=1}^N V(S_iT,\eps)=V(G,\eps).
  \end{align}
   Taking into account that $|h(\eps)|$ is bounded by some absolute constant on any fixed interval $[a,b]$ with $a>0$ (e.g., by $(N+1) \lambda_d(T)$), the validity of the estimate \eqref{eqn:VG} for $h$ is just a matter of adapting the constant $c$.
 % Now the hypothesis $\dim_M(\bd G, G)<D$ implies that there is some constant $\gamma>0$ such that $\dim_M(\bd G, G)<D-\gamma<D$ implying that $\eps^{D-d-\gamma}V(G,\eps)$ is uniformly bounded as $\eps \searrow 0$. This completes the proof.
\end{proof}
\begin{remark} \label{rem:renewal}
  It is an easy exercise to check that, for any $a>0$,
  $$
  \int_0^a \eps^{D-d-1} \varphi_{d,a}(\eps)\ d\eps= \int_0^1 \eps^{D-d-1} \varphi_{d,1}(\eps)\ d\eps,
  $$
  where $\varphi_{d,a}(\eps):=f(\eps)-\sum_{i=1}^N r_i^{d} \ind{(0,r_ia]}(\eps)  f(\eps/r_i)$, $\eps>0$. This clarifies that formula \eqref{eqn:Yd} is valid for an arbitrary constant $g$ (not just for the inradius), provided the same constant $g$ is used in the indicator functions in the definition of $h$. It was just a convenience, to use the inradius $g$ as $a$ in the proof.
  Since the continuity properties of $\varphi_{d,a}$ are the same for any $a>0$ and since an estimate of the type $|\varphi_{d,a}(\eps)|\le c \eps^{d-D+\gamma}$ holds either simultaneously for all $a>0$ or not at all, it is obvious that \cite[Theorem~4.1.4]{winter08} can be formulated with a general constant $a>0$ instead of $a=1$ inserted simultaneously in the indicator functions $\ind{(0,r_ia]}$ in the definition of $\varphi_d$ and as the upper bound of the integration interval in the conclusion.
\end{remark}

\begin{figure}%[b]
\begin{minipage}{120mm}
%\centering
  %\scalebox{0.5}{
  \includegraphics[width=120mm]{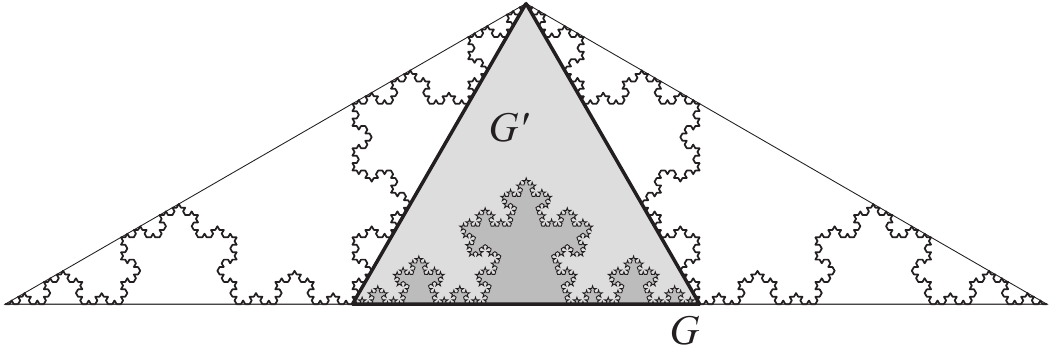}
 \end{minipage}
\begin{minipage}{65mm}
  %\centering
  %\scalebox{0.5}{
 % \includegraphics[width=70mm]{pics/pict3-5b}
\end{minipage}
\caption{\label{fig:KC} The interior $O$ of its convex hull is a feasible open set for the Koch curve. The generator of the tiling $\sT(O)$ is the gray equilateral triangle $G$. If $G$ is replaced by the open subset $G'$ of $G$ above the fractal curve, another tiling $\sT(O')$ of some other feasible set $O'$ is generated, cf.~Example~\ref{ex:Koch}.}
\end{figure}

\begin{exm}(\emph{A tiling with $\dim_M (\bd G,G)>D$.})\label{ex:Koch}
Let $F$ be the Koch curve (generated with two mappings) and let $O$ be the interior of its convex hull. Let $G$ be the generator of the tiling $\sT(O)$. $G$ is the grey equilateral triangle in Figure~\ref{fig:KC}. Now we construct a new tiling by modifying $G$ as follows: Replace the base of $G$ by some Koch type curve of dimension $\Delta> D$ and let $G'$ be the open set bounded by this curve and the remaining two sides of $G$. Take $G'$ as the generator of a new tiling. It is easy to see that the set $T':=\bigcup_{\sigma\in\Sigma_N^*} S_\sigma G'$ is a feasible open set for the Koch curve and that $G'$ is the generator of the tiling $\sT(T')$. Note that $T'$ is also the union set of the tiles. Obviously, $\dim_M(\bd G', G')\geq\Delta$ and thus also $\dim_M(\bd T',T')\geq\Delta>D$.
\end{exm}

%\begin{rem}
%  The hypothesis on the generator $G$ is not only sufficient it is close to being necessary. It is easy to construct sets $G$ with $\dim_M (\bd G,G)\geq D$, in which case the estimate \eqref{eqn:phi_d} is obviously not satisfied and the Renewal Theorem cannot be applied. (Add example!) On the other hand, at least for $\dim_M(\bd G, G)>D$, also the conclusion is not true. In this case, (because of the set inclusions $G\subset T$ and $\bd G\subset \bd T$) the Minkowski dimension of the $\sT$ is also strictly larger than $D$ -- it is at least $\dim_M(\bd G, G)$.
%\end{rem}

%  \cite[Thm.~6.2]{GeometryOfSST}, we have $G_{-\eps}\subset F_\eps$ (since $\bd G\subset F$). Moreover, we have $G\subset (\mathbf{S} O)^c$, implying $G_{-\eps}\subset((\mathbf{S} O)^c)_\eps$.
%  Now it follows from \cite[Cor. 5.6.3]{winter08} that there exist $c,\gamma>0$ such that for all $0<\eps<r_{\min}$, we have
%  \begin{align}
%    \varphi_d(\eps)=V(G,\eps)\le \lambda_d(G_{-\eps})\le\lambda_d(F_\eps\cap((\mathbf{S}O)^c)_\eps)\le c\eps^{d-D+\gamma}
%  \end{align}
%  This clearly
%implies that $\dim_M(\bd G, G)\le D-\gamma<D$.

%In particular, we discuss tilings with monophase and pluriphase generators as well as \emph{compatible tilings} (i.e., those satisfying the compatibility condition $\bd O\subset F$), which allow to derive results on the Minkowski contents of the associated self-similar set. An essential common point is that in all these special situations the hypothesis $\dim_M(\bd G, G)<D$ is satisfied.

\paragraph{{\bf Reformulation in terms of the generator.}}
%Before we explore the consequences of the above result in different special situations, we derive an alternative general formula for the (average) Minkowski content of a tiling in terms of its generator.
In the setting of Theorem~\ref{thm:main}, the only assumption on the generator is that $\dim_M (\bd G, G)<D$.
%, that is, the Minkowski dimension of the boundary of the generator $G$ (relative to $G$) is strictly less than the Minkowski dimension $D$ of $F$.
In this general situation one cannot expect formulas as explicit as the ones derived in the monophase case in \cite{LPW2}. But as in this case, the (average) Minkowski content of $\sT$ can be described completely in terms of the generator $G$ (and the contraction ratios of the IFS). To demonstrate this, we start with the expression derived in Theorem~\ref{thm:main} for the average Minkowski content of $\sT$ (as well as for its Minkowski content in the nonlattice case). Rearranging the integrals slightly and using the second equality in \eqref{eqn:h-G-relation}, we get:
\begin{align}
   \eta \asM^D(\sT) &= \int_0^g \eps^{D-d-1} h(\eps)\ d\eps= \int_0^g \eps^{D-d-1} \left(V(T,\eps)-\sum_{i=1}^N \ind{(0,r_i g ]}(\eps) V(S_i T, \eps)\right)\ d\eps \notag\\
   &= \int_0^g \eps^{D-d-1} \left(V(T,\eps)-\sum_{i=1}^N V(S_i T, \eps)+ \sum_{i=1}^N \ind{(r_i g,g]}(\eps) V(S_i T, \eps)\right)\ d\eps \notag\\
   &= \int_0^g \eps^{D-d-1} V(G,\eps)\ d\eps +  \sum_{i=1}^N  \int_{r_i g}^g \eps^{D-d-1} V(S_i T, \eps)\ d\eps\,. \label{eqn:Yd2}
\end{align}
 Now we apply the substitution $\tilde\eps=\eps/r_i$ to the $i$-th integral in the second sum. Employing that $V(T,\tilde\eps)=V(T,g)$ for $\tilde\eps\geq g$, we obtain
\begin{align}
  \eta \asM^D(\sT) &= \int_0^g \eps^{D-d-1} V(G,\eps)\ d\eps + \frac{V(T,g)}{d-D} g^{D-d} \left(1- \sum_{i=1}^N r_i^d \right)\\
   &= \int_0^g \eps^{D-d-1} V(G,\eps)\ d\eps + \frac{V(G,g)}{d-D} g^{D-d},
\end{align}
where we have used the relation $V(T,g)=V(G,g)\zeta_\sL(d)=V(G,g)(1- \sum_{i=1}^N r_i^d)^{-1}$ for the second equality, cf.\ e.g.~\cite[Def.~3.1]{LPW2}. Now observe that
\begin{align}
  \int_g^\infty V(G,\eps)\eps^{D-d-1}\ d\eps= \frac{V(G,g)}{d-D} g^{D-d},
\end{align}
which allows to derive the following elegant formula for the (average) Minkowski content of an arbitrary self-similar tiling $\sT$.
\begin{cor} \label{cor:main}
Let $\sT=\sT(O)$ be a self-similar tiling generated on a feasible set $O$ such that the generator $G$ satisfies $\dim_M(\bd G, G)<D$.
%, where $D$ is the Minkowski dimension of the associated self-similar set $F$.
Then the average Minkowski content is determined completely in terms of the volume function $V(G,\eps)$ of the generator $G$ via the formula
   \begin{align} \label{eqn:mink-gen}
  \asM^D(\sT)&=\frac 1\eta \int_0^\infty \eps^{D-d-1} V(G,\eps)\ d\eps.
\end{align}
In the nonlattice case, also the Minkowski content of $\sT$ is given by \eqref{eqn:mink-gen}.
\end{cor}

\begin{remark}
  Note that $\dim_M (\bd B, B)\geq d-1$ for any nonempty bounded open set $B\subset\R^d$. Hence, the assumption $\dim_M(\bd G, G)<D$ in Theorem~\ref{thm:main} and Corollary~\ref{cor:main} implies in particular, that $D>d-1$. This is an assumption made in earlier work on self-similar tilings, see e.g.~\cite{LPW1,LPW2}. It is now naturally implied by the present hypothesis on $G$. The parallel sets of self-similar tilings are not so well suited to the study of self-similar sets of dimensions smaller than $d-1$. However, this limitation can be overcome by replacing the parallel volume $V(G,\eps)$ with the function $\lambda_d(F_\eps\cap G)$ for which similar results are derived below in Theorems~\ref{thm:Minkowski-measurability-fractals2} and \ref{thm:main2}.
  \end{remark}

\paragraph{{\bf Monophase generators.}} Now we look at the situation studied in \cite{LPW2} when the generator of the tiling is monophase. Recall that a set is called monophase, if its inner parallel volume has a polynomial representation:
\begin{align}\label{eqn:def-prelim-Steiner-like-formula}
    V(\gen,\ge) = \sum_{k=0}^{d-1} \gk_k(\gen) \ge^{d-k},
    \qq\text{ for } 0 < \ge \leq \genir,
  \end{align}
  where the coefficients $\gk_k(G)$ are some real numbers depending only on $G$.

We point out, that in \cite{LPW2} the generator was assumed to be connected, which is not necessary. More precisely, in \cite{LPW2} the connected components of the set $G=O\setminus \setmap K$ are regarded as the generators, allowing a tiling to have more than one generator, and the results are restricted to tilings with a single connected generator. However, the connectedness is not used in the proofs. The results in \cite{LPW2} remain true for any (that is, not necessarily connected) monophase generator.

The representation \eqref{eqn:def-prelim-Steiner-like-formula} implies that
\begin{align} \label{eqn:kappa_d-1}
  \sM^{d-1}(\bd G, G)=\lim_{\eps \searrow 0} \eps^{-1}V(G,\eps)=\lim_{\eps \searrow 0}\sum_{k=0}^{d-1} \gk_k(\gen) \ge^{d-k-1}=\kappa_{d-1}(G)
  %\ge^{d-1-d} V(\gen,\ge) \to \gk_{d-1}(G) \text{ as } \eps\searrow 0,
\end{align}
from which the relation $\dim_M (\bd G, G)\leq d-1$ is easily seen. The reverse inequality is true for any bounded open subset $G$ of $\bR^d$ and so we have $\dim_M(\bd G, G)= d-1$.
\begin{remark}
   The following argument shows that $\kappa_{d-1}(G)$ is strictly positive, which allows to conclude $\dim_M(\bd G, G)= d-1$ directly from \eqref{eqn:kappa_d-1} and provides an interpretation of this coefficient as the surface area of $G$: Observe that one has $V'(G,\eps)=\Ha^{d-1}(\bd G_{-\eps}\cap G)$ for each $\eps>0$ for which the derivative $V'(G,\eps)$ exists), see e.g.~\cite{RW09}. Since here $V(G,\eps)$ is a polynomial in $\eps$, $V'(G,\eps)$ exists for all $\eps>0$, and computing the derivative yields that
   $$
   \kappa_{d-1}(G)=\lim_{\eps \searrow 0} V'(G,\eps)=\lim_{\eps \searrow 0}\Ha^{d-1}(\bd G_{-\eps}\cap G).
   $$
  % Thus, as a limit of positive numbers which are uniformly bounded away from zero (e.g. for all $\eps\in(0,\eps_0)$ by the surface area of a ball inscribed in $G\setminus G_{-\eps_0}$), also $\kappa_{d-1}(G)$ must be strictly positive for any monophase generator. The above limit means that $\kappa_{d-1}(G)$ is the limit of the surface areas of inner $\eps$-parallel sets of $G$.

   %Alternatively, the $(d-1)$-dimensional Minkowski content of the boundary of a set $G$ is regarded as a surface measure for $G$. The relative content $\sM^{d-1}(\bd G, G)$ can be seen as a one-sided (inner) version. Since

   %Observe that \eqref{eqn:kappa_d-1} means that $\kappa_{d-1}(G)$ is also the surface area of the generator $G$ in the sense of the (inner) Minkowski content.
%   %Moreover, we have
%   $$
%   \sM^{d-1}(\bd G, G)=\lim_{\eps \searrow 0} \eps^{-1}V(G,\eps)=\lim_{\eps \searrow 0}\sum_{k=0}^{d-1} \gk_k(\gen) \ge^{d-k-1}=\kappa_{d-1}(G),
%   $$
%   that is $\kappa_{d-1}(G)$ is also of the generator $G$.
   %If $\bd G$ is sufficiently smooth, then these surface areas will converge to the surface area of $G$, see ..., which means that $\kappa_{d-1}(G)$ is the surface area of $G$. However, this does not happen in general! ...
\end{remark}

Thus, assuming that $D>d-1$ (an assumption present in all results concerning tube formulas for self-similar tilings and in particular in the Minkowski measurability results obtained in \cite{LPW2}),  the hypothesis $\dim_M(\bd G, G)<D$ is satisfied and Theorem~\ref{thm:main} and Corollary~\ref{cor:main} apply.
Combining \eqref{eqn:Yd2} of Corollary~\ref{cor:main} with the representation \eqref{eqn:def-prelim-Steiner-like-formula}, we get the following expression for the average Minkowski content of $\sT$ (as well as for its Minkowski content in the nonlattice case):
\begin{align} \label{eqn:Mink-monophase}
   \eta \asM^D(\sT) &= \int_0^g \eps^{D-d-1} V(G,\eps)\ d\eps+\int_g^\infty \eps^{D-d-1} V(G,\eps)\ d\eps\notag\\
   &= \int_0^g \eps^{D-d-1} \sum_{k=0}^{d-1} \gk_k(\gen) \ge^{d-k}\ d\eps+\int_g^\infty \eps^{D-d-1} V(G,g)\ d\eps \notag\\
   &= \sum_{k=0}^{d-1} \frac{\kappa_k(G)}{D-k} g^{D-k} +  \frac{V(G,g)}{d-D} g^{D-d},
\end{align}
where for the last integral we used that, for $\tilde\eps\geq g$, one has $V(G,\tilde\eps)=V(G,g)$.

%Now we recall that $V(T,g)=V(G,g)\zeta_\sL(d)$, where $\zeta_\sL(s)=\sum_{l\in\sL} l^s, s\in\bC, \Re(s)>D$ is the scaling zeta function of the associated fractal string $\sL=\{r_w: w\in\bigcup_{n\in\bN_0}\{1,\ldots,N\}^n$.
%It is well known that $\zeta_\sL$ is given by (cf.~e.g. \cite[eq.~(5.9)]{LPW1} or \cite[Def.~3.1]{LPW2})
%\begin{align}\label{eqn:gzL-extended-to-C}
%    \gzL(s) = \frac1{1 - \sum_{i=1}^N r_i^s},
%    \text{ for } s\in\bC.
%  \end{align}
%Plugging this into the second term of the last expression, $\zeta_\sL(d)$ cancels with the term $-(\sum_{i=1}^N r_i^d -1)$ and we get
%\begin{align*}
%   \eta \asM^D(\sT) &= \sum_{k=0}^{d-1} \frac{\kappa_k(G)}{D-k} g^{D-k} +  \frac{V(G,g)}{d-D} g^{D-d}.
%\end{align*}
Using again the representation \eqref{eqn:def-prelim-Steiner-like-formula} for $\eps=g$ and combining the coefficients with the same $k$, we arrive at
\begin{align*}
   \eta \asM^D(\sT)&= \frac 1{d-D}\sum_{k=0}^{d-1} \frac{d-k}{D-k} \kappa_k(G) g^{D-k}=\frac 1{d-D}\Gamma_D(G),
\end{align*}
where the function $\Gamma_s(G)$, defined in \cite[Def.~4.6]{LPW2}, is given by
\begin{align}\label{eqn:Md(G)}
    \Gamma_s(\gen)
    := \sum_{k=0}^{d-1} \frac{d-k}{s-k} \kappa_k(\gen)\genir^{s-k}, s\in\bC.
  \end{align}
Hence we have proved the following statement:
\begin{cor} \label{cor:monophaseR}
Let $\sT=\sT(O)$ be a self-similar tiling in $\R^d$ with $D>d-1$. Assume that the generator $G$ is monophase.  %$\dim_M(\bd G, G)<D$.
%, where $D$ is the Minkowski dimension of the associated self-similar set $F$.
Then the average Minkowski content of $\sT$ is given by
 \begin{align}
  \label{eqn:Mink-mono}
  \asM^D(\sT)=\frac 1\eta \frac 1{d-D} \Gamma_D(G)\,.
\end{align}
In the nonlattice case, also the Minkowski content of $\sT$ is given by \eqref{eqn:Mink-mono}.
\end{cor}

Note that the right hand side of \eqref{eqn:Mink-mono} is precisely the expression derived in \cite[Theorem~4.8]{LPW2} for the (average) Minkowski content of $\sT$. (Observe that $\frac 1\eta= \res[D]{\zeta_\sL(s)}$, cf.~\cite[eq.~(4.7)]{LPW2}.) Thus, with the help of renewal theory, we have recovered in a rather simple way the results in \cite{LPW2} on the Minkowski measurability of self-similar tilings with a monophase (but not necessarily connected) generator, including the precise formula, except for the proof of the non-measurability in the lattice case.

%\begin{remark}
%   We emphasize that in the general case the proper replacement for the expression $\Gamma_G(D)$ appearing in the case of a monophase generator is the integral appearing in \eqref{eqn:mink-gen}, more precisely, we have the relation
%\begin{align}
%\int_0^\infty \eps^{D-d-1} V(G,\eps)\ d\eps=\frac 1{d-D} \Gamma_D(G),
%\end{align}
%in the monophase case. %which can also be verified directly using the representation \eqref{eqn:def-prelim-Steiner-like-formula}.
%This suggest to use the following geometric zeta function in the general case ...???
%\end{remark}

\medskip

\paragraph{{\bf Pluriphase generators.}} In \cite{TFCD} an open set $G$ (with inradius $g$) was called \emph{pluriphase}, if its inner parallel volume has a piecewise polynomial representation, that is, there exists a partition $0=\eps_0<\eps_1<\eps_2<\ldots<\eps_m=g$ such that
\begin{align}\label{eqn:def-pluriphase}
    V(\gen,\ge) = \sum_{k=0}^{d-1} \gk^\ell_k(\gen) \ge^{d-k},
    \qq\text{ for } \eps_{\ell-1} < \ge \leq \eps_\ell, \ell=1,\ldots,m,
  \end{align}
  where the coefficients $\gk^\ell_k(G)$ are some real numbers depending only on $G$. It is shown in \cite{Kocak2} that convex polytopes are pluriphase. Polytopes occur frequently as generators of self-similar tilings and constitute an important class of examples. (It is easily seen that also generators which consist of several (disjoint) convex polytopes are pluriphase.) Similarly as in the monophase case, one has
 $  \ge^{d-1-d} V(\gen,\ge) \to \gk^1_{d-1}(G)$  as $\eps \searrow 0$ implying again $\dim_M(\bd G, G)=d-1$.
Thus, provided $D>d-1$, the hypothesis $\dim_M(\bd G, G)<D$ is satisfied and Theorem~\ref{thm:main} and Corollary~\ref{cor:main} apply. Plugging the representation \eqref{eqn:def-pluriphase} into the general formula \eqref{eqn:mink-gen}, it is now a simple computation to derive the following formula for the average Minkowski content of a self-similar tiling $\sT$ with a pluriphase (but not necessarily connected) generator as well as for its Minkowski content in the nonlattice case:
\begin{align} \label{eqn:Mink-pluriphase}
    \eta\asM^D(\sT) &= \sum_{k=0}^{d-1} \frac 1{D-k} \left(\sum_{\ell=1}^{m-1}(\gk^\ell_k(G)-\gk^{\ell+1}_k(G))\eps_\ell^{D-k}+ \kappa^m_k(G) g^{D-k}\right) +  \frac{V(G,g)}{d-D} g^{D-d}.
\end{align}
Note that the case $m=1$ is the monophase case in which the above formula reduces to the formula in \eqref{eqn:Mink-monophase}.
Similarly as in the the monophase case, one can use that $V(G,g)=\sum_{k=0}^{d-1} \gk^m_k(G)g^{d-k}$ to incorporate the last term in formula \eqref{eqn:Mink-pluriphase} into the first sum and derive the equivalent formula \eqref{eqn:Mink-pluriphase2} below. Hence we have proved the following statement:
\begin{cor} \label{cor:pluriphase}
Let $\sT=\sT(O)$ be a self-similar tiling in $\R^d$ with $D>d-1$. Assume that the generator $G$ is pluriphase.  %$\dim_M(\bd G, G)<D$.
%, where $D$ is the Minkowski dimension of the associated self-similar set $F$.
Then the average Minkowski content of $\sT$ is given by
\begin{align} \label{eqn:Mink-pluriphase2}
 \asM^D(\sT) &= \frac 1\eta \sum_{k=0}^{d-1} \frac 1{D-k} \left(\sum_{\ell=1}^{m-1}(\gk^\ell_k(G)-\gk^{\ell+1}_k(G))\eps_\ell^{D-k}
   +  \frac{d-k}{d-D}\kappa^m_k(G) g^{D-k}\right).
\end{align}
In the nonlattice case, also the Minkowski content of $\sT$ is given by \eqref{eqn:Mink-pluriphase2}.
\end{cor}

%\begin{remark}
%  Formula \eqref{eqn:Mink-pluriphase2} for the Minkowski content in the pluriphase case is new. Now with the exact formula at hand, it should be easier to show that lattice pluriphase tilings are not Minkowski measurable with analogous arguments as used in \cite{LPW2} in the monophase case.
%\end{remark}

\medskip

\paragraph{\bf Compatible tilings and the Minkowski content of self-similar sets.} For a self-similar set $F\subset\R^d$ and a feasible set $O$ for $F$, let $\sT=\sT(O)$ denote the self-similar tiling generated on $O$. Write $K:=\cl{O}$. Assume that $\sT$ satisfies the \emph{compatibility condition}, that is, assume that $\bd G\subset F$ or, equivalently, that $\bd K\subset F$, see \cite[Theorem~6.2]{GeometryOfSST}. This condition has been shown in \cite[again Thm~6.2]{GeometryOfSST} to be necessary and sufficient for the following disjoint decomposition to hold for all $\eps>0$:
\begin{align} \label{eqn:compat2}
F_\eps=T_{-\eps}\cup (K_{\eps}\setminus K),
\end{align}
where $T=\bigcup_{R\in\sT} R$ is the union set of the tiling.
In case of compatibility, the tiling can be used directly to study the parallel volume of $F$. Some results on the Minkowski measurability of $F$ have been obtained in \cite[Thm.~5.4]{LPW2} using compatible tilings, see also \cite{Kom}, where the tiling idea is used implicitly. Here we strengthen these results by showing that for compatible tilings $\sT(O)$ constructed on a \emph{strong} feasible set $O$, the contribution of the sets $K_{\eps}\setminus K$ can be neglected and that the Minkowski contents of $F$ and $\sT$ always coincide. In particular, no assumption on the Minkowski measurability of $K$ is needed, and for the equality of the  Minkowski contents of $F$ and $\sT$ the generator need not be monophase.
%- except for the non-measurability in the lattice case - the eresults hold without the assumption on the generator to be monophase.
More precisely, we have the following results.
\begin{theorem}[Minkowski measurability of compatible self-similar fractals]
  \label{thm:Minkowski-measurability-fractals}
%\marginpar{is the condition $\Gamma_D(G)\neq 0$ really needed in Theorem~\ref{thm:Mink-meas}?; maybe some formula for $\Mink(F)$ and some statement on the average Minkowski content should be added ...}
Let $F$ be a self-similar set in $\bR^d$ with Minkowski dimension $D \in (d-1,d)$ satisfying OSC. %
%\footnote{See Remark~\ref{rem:coincidental-abscissae}.}
Assume there exists a strong feasible open set $O$ for $F$ such that
%\linenopax
%\begin{align} \label{eqn:compatibility}
 $ \bd O \subset F$, i.e.\ such that the generated tiling $\sT=\sT(O)$ is compatible.
%\end{align}
%and the closure $K:=\overline{O}$ of $O$ satisfies $\sS\sM_\abscissa(K) \in [0,\iy)$.%
%\footnote{In other words, either $K$ is (outer) Minkowski measurable of dimension $D$ or its $D$-dimensional outer Minkowski content vanishes; see Remark~\ref{rem:fractal-boundaries}.}
%  Furthermore, assume that the generator $G$ ($:=K\setminus \bigcup_{n=1}^N \simt_n (O)$) of the associated self-similar tiling $\tiling=\tiling(O)$ is monophase with $\Gamma_D(G)\neq 0$%
 % \footnote{In which case the proof of Theorem~\ref{thm:Minkowski-measurability-fractals} clearly shows that $\spam_\abscissa(\gen) > 0$.}
%  and (in the lattice case) that $\spam_{\abscissa+\ii m \per}(\gen) \neq 0$ for some $m \in \bZ\less\{0\}$, where \per is the oscillatory period as in Remark~\ref{rem:dichotomy}.

Then the ($D$-dim.) outer Minkowski content of $K=\overline{O}$ is zero, i.e.\ $\sM^D(K,K^c)=0$. Moreover, the average Minkowski contents of $F$ and $\sT$ coincide, i.e.,
\begin{align} \label{eq:asM-F-T}
  \asM^D(F)=\asM^D(\sT),
\end{align}
and both are given by the finite and positive expression in \eqref{eqn:mink-gen} (as well as by \eqref{eqn:Yd}). %$F$ has Minkowski dimension \abscissa, and

Furthermore, $F$ is Minkowski measurable if and only if $\sT$ is Minkowski measurable. In this case, the Minkowski content of $F$ (as well as that of $\sT$) is again given by the expression in \eqref{eqn:mink-gen}.
\end{theorem}
Combining this result with \cite[Theorem~4.8]{LPW2}, we obtain the following strengthening of \cite[Theorem~5.4]{LPW2} in the monophase case.
\begin{cor} \label{cor:monophase}
  If in addition to the hypothesis in Theorem~\ref{thm:Minkowski-measurability-fractals}, the generator $G$
  is assumed to be monophase, then the set $F$ is Minkowski measurable if and only if it is nonlattice.
\end{cor}
\begin{proof}
  In \cite[Theorem~4.8]{LPW2}, it is shown that a self-similar tiling $\sT$ (with a monophase generator $G$) is Minkowski measurable if and only if it is nonlattice. Hence, by Theorem~\ref{thm:Minkowski-measurability-fractals}, the same must hold for $F$.
\end{proof}

The proof of Theorem~\ref{thm:Minkowski-measurability-fractals} relies heavily on the following estimate obtained in \cite{winter08} for strong open sets:
\begin{lemma}[{\cite[Corollary~5.6.3]{winter08}}] \label{lem:key}
Let $O$ be a strong feasible open set of an IFS $\{S_1,\ldots,S_N\}$ in $\R^d$ with similarity dimension $D$. Then there exist some constants $c,\gamma>0$ such that, for all $\eps\in(0,1)$,    \begin{align}\label{eqn:V-est}
   \lambda_d(F_\eps\cap ((\setmap O)^c)_\eps)\leq c\eps^{d-D+\gamma}.
   \end{align}
\end{lemma}

\begin{proof}[Proof of Theorem~\ref{thm:Minkowski-measurability-fractals}]
   First observe that $\setmap O\subset O$ implies $K_\eps\setminus K\subset O^c\subset(\setmap O)^c\subset((\setmap O)^c)_\eps$ for any $\eps>0$. Moreover, the compatibility assumption %\eqref{eqn:compatibility}
   implies $F_\eps\cap (K_\eps\setminus K)= K_\eps\setminus K$, cf.~\eqref{eqn:compat2}. Now we infer from Lemma~\ref{lem:key} (for which we need that $O$ is strong) and the above set inclusions, that the estimate \eqref{eqn:V-est} holds equally for $\lambda_d(K_\eps\setminus K)$ from which it follows immediately that $\sM^D(K,K^c)=0$ as claimed.

   Now observe that, by \eqref{eqn:compat2}, we have for all $\eps>0$,
    \begin{align*} %\label{eqn:compat3}
     \eps^{D-d}\lambda_d(F_\eps)=\eps^{D-d}\lambda_d(T_{-\eps})+\eps^{D-d}\lambda_d(K_{\eps}\setminus K).
      \end{align*}
Taking the limit on both sides as $\eps \searrow 0$ and recalling that the second term on the right does always tend to zero, we conclude that the limit on the left hand side (that is, $\sM^D(F)$) exists if and only if the limit of the first term on the right hand side (that is, $\sM^D(\sT)$) exists, and that both numbers coincide in this case. The claimed expressions for the Minkowski contents follow immediately from Theorem~\ref{thm:main} and Corollary~\ref{cor:main}, once we have verified that $\dim_M(\bd G, G)<D$. For this we employ again Lemma~\ref{lem:key}. Observe that $G$ satisfies the inclusion $G\subset (\setmap O)^c$, which implies $G_{-\eps}\subset (\setmap O)^c\subset ((\setmap O)^c)_\eps$. Moreover, we have $G_{-\eps}=F_\eps\cap G\subset F_\eps$ from the decomposition \eqref{eqn:compat2}. Together this gives $\lambda_d(G_{-\eps})\leq \lambda_d(F_\eps\cap ((\setmap O)^c)_\eps)$ and from \eqref{eqn:V-est} we conclude
that $\eps^{\alpha-d}\lambda_d(G_{-\eps})\to 0$ for any $\alpha>D-\gamma$. Hence $\dim_M(\bd G, G)\leq D-\gamma<D$. This shows that the hypotheses of Theorem~\ref{thm:main} and Corollary~\ref{cor:main} are satisfied and the remaining assertions follow from these results.
\end{proof}

\label{page11} The assumption that $O$ is a \emph{strong} feasible set cannot easily be omitted in Theorem~\ref{thm:Minkowski-measurability-fractals}. Indeed the contribution of the outer parallel set $K_\eps\setminus K$ may be positive, as the following example shows. One can however get rid of the compatibility assumption by changing the point of view on the tilings, see the next paragraph and particularly Theorem~\ref{thm:Minkowski-measurability-fractals2} below.

\begin{figure}%[b]
\begin{minipage}{63mm}
%\centering
  %\scalebox{0.5}{
  \includegraphics[width=63mm]{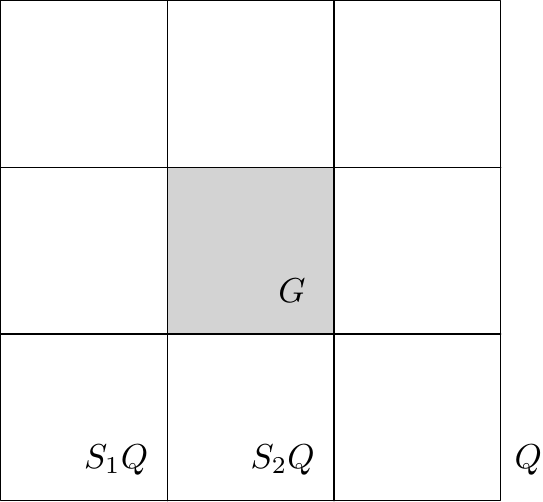}
 \end{minipage}
\begin{minipage}{60mm}
  %\centering
  %\scalebox{0.5}{
  \includegraphics[width=60mm]{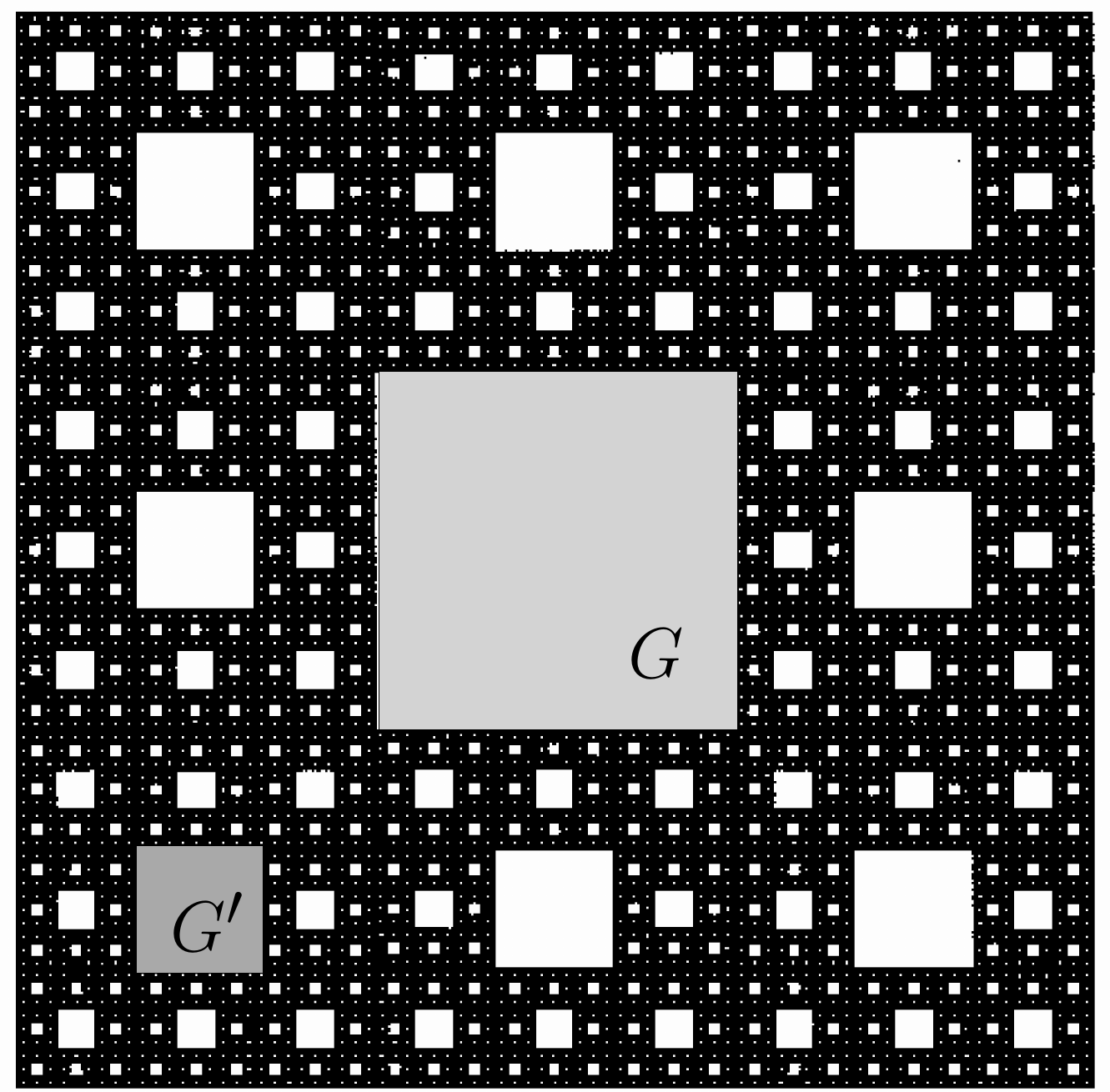}
\end{minipage}
\caption{\label{fig:SC1}(Left): The interior $O$ of the unit square $Q$ is a feasible open set for the Sierpinski carpet. The gray square $G$ in the middle is the generator of the tiling $\sT(O)$ in Example~\ref{ex:SC}. (Right): The gray set $G'$ is the generator of the tiling $\sT'=\sT(O')$ of some other feasible set $O'$ in Example~\ref{ex:SC} which is not strong.}
\end{figure}

\begin{exm} \label{ex:SC}

  Let $F\subset\R^2$ be the standard Sierpinski carpet generated by eight similarities each mapping the unit square $Q:=[0,1]^2$ to one of the eight subsquares shown in Figure~\ref{fig:SC1}. It is easily checked that the set $O=\inn(Q)$ is a strong feasible open set for $F$. The generator $G$ of the tiling $\sT=\sT(O)$ is the (open) gray middle square in Figure~\ref{fig:SC1}. Since $\bd G\subset F$, $\sT$ is compatible and hence, by Theorem~\ref{thm:Minkowski-measurability-fractals}, $\sM^D(F)=\sM^D(\sT)$.

  Now let $G':=S_1 G$ and $O':=\bigcup_{\sigma\in\Sigma_N^*} S_\sigma G'$. It is easily checked that $O'$ is a feasible open set for $F$ and that $O'\cap F=\emptyset.$ Hence $O'$ is not strong. The set $G'$ is the generator of the tiling $\sT'=\sT(O')$ constructed on $O'$.
  We claim that $\sM^D(\sT')<\sM^D(\sT)$. Indeed, applying Corollary~\ref{cor:main} and noting that
  $$
  V(G',\eps)=V(\frac13 G,\eps)=3^{-2} V(G,3\eps),
  $$
  for any $\eps>0$, we get
  \begin{align*}
    \eta \sM^D(\sT')&=\int_0^\infty \eps^{D-2-1} V(G',\eps) d\eps=\int_0^\infty \eps^{D-2-1} 3^{-2} V(G,3\eps) d\eps\\
    &=\int_0^\infty \left(\delta/3\right)^{D-2-1} 3^{-2} V(B,\delta) d\delta=\left(\frac 13\right)^{D-1} \eta \sM^D(\sT)\\
    &=\frac38 \eta \sM^D(\sT).
  \end{align*}
  Hence we have constructed a tiling $\sT'$ for the Sierpinski carpet $F$ such that $\sM^D(\sT')\neq\sM^D(F)$, showing that the conclusion of Theorem~\ref{thm:Minkowski-measurability-fractals} may fail  without the assumption that the feasible set is strong.

  To complete the picture, we point out that on the other hand it is not necessary to have a strong feasible set for the conclusion of Theorem~\ref{thm:Minkowski-measurability-fractals} to hold. The set $O'':=\bigcup_{\sigma\in\Sigma_N^*} S_\sigma G$ is a feasible set for $F$ which is not strong. It generates the same tiling as the set $O$, i.e. $\sT(O)=\sT(O'')$. Hence one has in particular $\sM^D(\sT(O''))=\sM^D(\sT(O))=\sM^D(F)$.
\end{exm}

Generalizing the construction of the set $O''$ in Example~\ref{ex:SC}, one can say that for each strong feasible set $O$ generating a compatible tiling $\sT=\sT(O)$ there exists a feasible set $O''$ which is not strong and which generates the same compatible tiling. It is not clear, whether the converse is also true: Given an arbitrary feasible set $\tilde O$ such that $\sT(\tilde O)$ is compatible, does there exist a \emph{strong} feasible set $O$ such that $\sT(O)$ is  compatible?

\begin{remark} \label{rem:compatible-limit}
It is well known that not all self-similar sets possess a feasible set such that the generated tiling is compatible. Indeed, it is shown in \cite[cf.~Theorem 7.2]{PokWi} that a self-similar set $F\subset\R^d$ (satisfying OSC and $\dim_M F<d$) possesses a compatible tiling if and only if the complement of $F$ is disconnected. Simple (self-similar) curves like the Koch curve are for instance not compatible.
So even if one was able to give a positive answer to the question raised above, the applicability of the results above would be limited to sets with a disconnected complement, excluding in particular all self-similar sets in $\R^d$ of dimension $D<d-1$. The alternative approach discussed below overcomes these limitations.
\end{remark}

\medskip

\paragraph{\bf Generator formulas for arbitrary self-similar sets.} We now suggest a slightly different approach to computing the Minkowski content of \emph{self-similar sets} using tilings which allows a simple and complete  geometric characterization of those self-similar sets for which a generator-type formula exists. We consider the tiling as a way to conveniently \emph{partition} the parallel sets of $F$. Instead of the parallel volume $V(T,\eps)$ of the union set $T$ of the tiling we will study the parallel volume $\lambda_d(F_\eps\cap T)$ of $F$ restricted to $T$, which amounts to studying the relative Minkowski content $\sM^D(F,T)$ of $F$ relative to $T$. It is even more convenient, to look instead at the relative Minkowski content $\sM^D(F,O)$ of $F$ relative to the feasible set $O$ itself. (Note that $\cl{T}=\cl{O}$.) This approach allows in fact to obtain generator type formulas for all self-similar sets for which tilings exists. We first state the main result:
\begin{theorem} \label{thm:main-main}
  Let $F$ be a self-similar set in $\R^d$ satisfying OSC and let $D<d$.

  Then there exists a strong feasible open set $O$ for $F$ such that
 the average Minkowski content (and, in case it exists, also the Minkowski content) of $F$ is given by the formula
\begin{align} \label{eqn:main-main}
  \asM^D(F)&=\frac 1\eta \int_0^\infty \eps^{D-d-1} \lambda_d(F_\eps\cap \Gamma)\ d\eps
\end{align}
where $\Gamma:=O\setminus \setmap O$.
\end{theorem}
%We will provide a necessary and sufficient condition for a \emph{generator formula} to hold.
The essential point is that in order to provide such a generator formula, a feasible set $O$ needs to satisfy the following \emph{projection condition}:
\begin{align}\label{eqn:PC}
\tag{PC}  S_i O\subset \cl{\pi^{-1}_F(S_i F)}, \quad \text{ for } i=1,\ldots,N.
\end{align}
Here $\pi_F$ denotes the metric projection onto the set $F$. Note that $\pi_F$ is defined on the set $U(F)$ of points $x\in\R^d$ which have a unique nearest point in $F$.

Our first step now is to show that there exists always a strong feasible set $O$ satisfying \eqref{eqn:PC} by giving an explicit example of such an $O$. In \cite{BHR}, the \emph{central open set} $V_c$ of an IFS is introduced and shown to be feasible. Erin Pearse observed that $V_c$ is, in fact, strong and satisfies \eqref{eqn:PC}. To recall the definition of $V_c$, let $\phi$ be the empty word
%set of all finite words of length at least 1 a
and let $\sH:=\{h=S_\sigma^{-1}S_\omega: \sigma,\omega\in\Sigma_N^{*}\setminus\{\phi\}, \sigma_1\neq\omega_1\}$ be family of all \emph{neighbor maps} of the IFS $\{S_1,\ldots,S_N\}$. Here $\sigma_1,\omega_1$ are the first letters of the finite words $\sigma, \omega$, respectively. Let $H:=\bigcup\{h(F): h\in\sH\}$ be the union of all \emph{neighbors} of $F$.
Then the central open set is defined by
\begin{align}
  \label{eq:Vc} V_c:=\{x\in\R^d: d(x,F)<d(x,H)\}.
\end{align}
\begin{prop}
  \label{prop:Vc} For any self-similar set $F\subset\R^d$ satisfying OSC, the central open set $V_c$ is a strong feasible set satisfying the projection condition \eqref{eqn:PC}.
\end{prop}
\begin{proof}
   It is shown in \cite[Theorem 1]{BHR}, that $V_c$ is feasible. This implies in particular that $V_c$ is nonempty. Let $y\in V_c$. Then there exists a point $x\in F$ such that $d(y,x)=d(y,F)$. Since the assumption $0=d(x,F)=d(x,H)$ implies $d(y,H)=d(y,F)$ which contradicts $y\in V_c$, it follows that $d(x,F)<d(x,H)$ and thus $x\in V_c$. Hence $V_c\cap F\neq\emptyset$, i.e., $V_c$ is strong.

   To see that $V_c$ satisfies \eqref{eqn:PC}, let $x\in S_i V_c$ for some $i\in\{1,\ldots,N\}$. Then there exists a point $v\in V_c$ such that $S_iv=x$. Since $S_i^{-1} S_jF\subset H$ for $j\neq i$, we have
   $$
   d(v,F)<d(v,H)\leq d(v, {\bigcup}_{\substack{j=1\\ j\neq i}}^N S_i^{-1}S_j F)=d(v, S_i^{-1}{\bigcup}_{\substack{j=1\\ j\neq i}}^N S_j F),
   $$
   and applying $S_i$ yields
   $
   d(x, S_i F)<d(x, \bigcup _{j\neq i} S_j F).
   $
  Hence $\pi_F(x)\in S_i F$, which shows that \eqref{eqn:PC} holds.
\end{proof}

Thus we can always find a strong feasible set such that the projection condition holds.
We note that due to its construction $V_c$ may be a rather complicated set and often one can find simpler strong feasible sets satisfying the projection condition. In particular, we will detail later that for any strong feasible $O$ which produces a compatible tiling the projection condition is satisfied, see Remark~\ref{rem:compat-recover}.% as well as $\lambda_d(\bd G)=0$.
This allows to recover some of the previous results for the compatible case from the following general statement, which describes the precise relation between generator formulas and the projection condition and from which also Theorem~\ref{thm:main-main} is easily derived. Recall that $\sM^D(F,O)$ is the Minkowski content of $F$ relative to $O$, cf.\ \eqref{eqn:inner_Mink_def}, and $\asM^D(F,O)$ denotes the averaged counterpart of $\sM^D(F,O)$, cf.\ \eqref{eq:avMC}.

\begin{theorem}[Minkowski content of self-similar fractals - general case] \label{thm:Minkowski-measurability-fractals2}
Let $F$ be a self-similar set in $\bR^d$ satisfying OSC. Denote by $D$ its similarity dimension and let $O$ be an arbitrary strong feasible open set for $F$.
\begin{enumerate}[(i)]
%\footnote{See Remark~\ref{rem:coincidental-abscissae}.}
\item Then the relative ($D$-dimensional) Minkowski content of $F$ relative to $O^c$ is zero, i.e.\ $\sM^D(F,O^c)=0$. As a consequence, the average Minkowski content of $F$ coincides with the average relative Minkowski content of $F$ relative to $O$, i.e.
    $$\asM^D(F)=\asM^D(F,O).$$
    Furthermore, $F$ is ($D$-dimensional) Minkowski measurable if and only if $\sM^D(F,O)$ exists and is positive and finite. In this case, $\sM^D(F)=\sM^D(F,O)$.
\item Let $\Gamma:=O\setminus \setmap O$. % and assume that $\lambda_d(\bd G)=0$.
Then the average Minkowski content (and, in case it exists, also the Minkowski content) of $F$ is given by the formula
\begin{align} \label{eqn:rel-mink-gen}
  \asM^D(F)&=\frac 1\eta \int_0^\infty \eps^{D-d-1} \lambda_d(F_\eps\cap \Gamma)\ d\eps
\end{align}
if and only if $D<d$ and the projection condition \eqref{eqn:PC} holds.
\item Let $G:=O\setminus \cl{\setmap O}$. The assertion in (ii) remains true with the set $\Gamma$ in \eqref{eqn:rel-mink-gen} replaced by $G$, provided $\lambda_d(\bd G)=0$.
\end{enumerate}
\end{theorem}

\begin{proof}[Proof of Theorem~\ref{thm:main-main}]
  By Proposition~\ref{prop:Vc}, there exists a strong feasible set $O$ satisfying \eqref{eqn:PC}. Therefore, the assertion of Theorem~\ref{thm:main-main} follows by applying part (ii) of Theorem~\ref{thm:Minkowski-measurability-fractals2} to this set $O$.
\end{proof}
For the proof of part (ii) and (iii) of Theorem~\ref{thm:Minkowski-measurability-fractals2}, we require the following statement, which is a  counterpart of Theorem~\ref{thm:main} %and Corollary~\ref{cor:main}
for the function $\lambda_d(F_\eps\cap O)$ (instead of $V(T,\eps)$). For sets $F$ with $D<d$, the role of the inradius $g$ of the generator $G$, is now taken by following \emph{relative inradius}:
\begin{align} \label{eq:rel-inr}
  \tilde{g}:=\sup\{d(x,F):x\in G\}.
\end{align}
Note that $\tilde{g}$ is equivalently given by $\sup\{d(x,F):x\in O\}$. Indeed, one inequality is obvious from the inclusion $G\subset O$ and for the reverse inequality note that the tiling $\sT(O)$ exists in this case. For $x\in T$, we have $x\in S_\omega G$ for some $\omega\in\Sigma_N^*$ and thus $d(x,F)\leq d(x,S_\omega F)=r_w d(S_\omega^{-1}(x), F)\leq\tilde g$, since $S_\omega^{-1}(x)\in G$. For general $x\in O$, the relation $d(x,F)\leq \tilde g$ follows now from $\cl{O}=\cl{T}$.

\begin{prop} \label{thm:main2}

Let $F\subset\R^d$ be a self-similar set satisfying OSC and $\dim_M F=D<d$. Let
%$\sT=\sT(O)$ be a self-similar tiling generated on a
$O$ be a strong feasible open set. % satisfying \eqref{eqn:PC}.
%Assume that the generator $G$ of $\sT$ satisfies $\dim_M(F, G) <D$.
%, where $D$ is the similarity dimension of the associated IFS.
Then the $D$-dimensional average Minkowski content $\widetilde\sM^D(F,O)$ of $F$ relative to $O$ exists and coincides with the strictly positive value
\begin{align}\label{eqn:Yd-2}
\frac{1}{\eta} \int_0^{\tilde{g}} \eps^{D-d-1} \varphi(\eps)\ d\eps,
\end{align}
where $\eta= \sum_{i=1}^N r_i^D |\ln r_i|$ and the function $\varphi:(0,\infty)\to \bR$ is given by
\begin{equation} \label{eqn:h-def-2}
\varphi(\eps) := \lambda_d(F_\eps\cap O)-\sum_{i=1}^N \ind{(0,r_i{\tilde{g}}]}(\eps) \lambda_d((S_iF)_\eps\cap S_iO)\,.
\end{equation}
%\begin{align} \label{eqn:mink-gen2}
%  \asM^D(F, O)&=\frac 1\eta \int_0^\infty \eps^{D-d-1} \lambda_d(F_\eps\cap G)\ d\eps.
%\end{align}
If $F$ is nonlattice, then also the Minkowski content $\sM^D(F, O)$ of $F$ relative to $O$ exists and equals the expression in \eqref{eqn:Yd-2}.
\end{prop}

\begin{proof}%[Proof of Proposition~\ref{thm:main2}]
 As in the proof of Theorem~\ref{thm:main}, we can either assume $\tilde g=1$ and apply \cite[Theorem~4.1.4]{winter08} directly to the functions
  $f(\eps):=\lambda_d(F_\eps\cap O), \eps>0$ and $\varphi$ as in \eqref{eqn:h-def-2} or apply the slight modification of this theorem discussed in Remark~\ref{rem:renewal} above.
 % \begin{equation} \label{h-def}
%\varphi(\eps) := \lambda_d(F_\eps\cap O)-\sum_{i=1}^N \ind{(0,r_i{\tilde{g}}]}(\eps) \lambda_d(F_\eps\cap S_i O)\,.
%\end{equation}
Note that
\begin{align} \label{eqn:V-S_iT2}
  \lambda_d((S_iF)_{\eps}\cap S_i O)=\lambda_d(S_i(F_{\eps/r_i}\cap O))=r_i^d \lambda_d(F_{\eps/r_i}\cap O)= r_i^d f(\eps/r_i).
\end{align}
Therefore, the following renewal equation holds for each $\eps>0$:
  \begin{equation} \label{eqn:renwaleqn}
\varphi(\eps)=f(\eps)-\sum_{i=1}^N r_i^{d} \ind{(0,r_i{\tilde{g}}]}(\eps)  f(\eps/r_i).
\end{equation}
It remains to show that the hypotheses on $\varphi$ in \cite[Theorem~4.1.4]{winter08} are satisfied. Since $f$ is continuous in $\eps$, it is obvious from \eqref{eqn:renwaleqn} that $\varphi$ is piecewise continuous with at most finitely many discontinuities. Furthermore, for  $\eps<\min_i r_i {\tilde{g}}$,
\begin{align*}
   \varphi(\eps)&= \lambda_d(F_\eps\cap O)-\sum_{i=1}^N  \lambda_d(F_\eps\cap S_iO)+\sum_{i=1}^N  \left\{\lambda_d(F_\eps\cap S_iO) -\lambda_d((S_iF)_\eps\cap S_iO)\right\}\\
   &=\lambda_d(F_\eps\cap \Gamma)+\sum_{i=1}^N  \lambda_d\left(\left(F_\eps\setminus(S_iF)_\eps\right)\cap S_iO\right),
\end{align*}
where $\Gamma:=O\setminus \setmap O$. %(Note that $G\subset \Gamma\subset \cl{G}$.)
 Now observe that $\Gamma\subset (\setmap O)^c\subset((\setmap O)^c)_\eps$ and $F_\eps\setminus(S_iF)_\eps\cap S_iO\subset((\setmap O)^c)_\eps$ for any $\eps>0$.
 Indeed, since $F\cap S_iO\subseteq S_iF$, for any point $x\in F_\eps\setminus(S_iF)_\eps\cap S_iO$ there exists a point $y\in F$ such that $d(x,y)\leq\eps$, $y\notin S_i F$ and thus $y\notin S_iO$. Therefore,
 $d(x, (\setmap O)^c)\leq d(x,(S_iO)^c)\leq d(x,y)\leq\eps$, which implies $x\in((\setmap O)^c)_\eps$ and proves the second claimed set inclusion.
 The inclusions allow to apply the estimate \eqref{eqn:V-est} of Lemma~\ref{lem:key} to each of the terms in the above sum (for which we use that $O$ is strong).
 We infer that there exist constants $\gamma, c>0$ such that, for each $0<\eps\le \min_i r_i {\tilde{g}}$,
\begin{equation}\label{eqn:VG3}
\varphi(\eps)\le c \eps^{d-D+\gamma}.
\end{equation}
Since, for ${\tilde{g}}\geq\eps\geq\min_i r_i {\tilde{g}}$, the function $\varphi$
    is bounded by some absolute constant (e.g., by $(N+1) \lambda_d(O)$), the estimate \eqref{eqn:VG3} holds for all $\eps\in(0,{\tilde{g}})$ (with the constant $c$ adapted if necessary).
   It follows now from \cite[Theorem~4.1.4]{winter08}, that $\asM^D(F,O)$ exists (and in the nonlattice case also $\sM^D(F,O)$) and is given by the expression in \eqref{eqn:Yd-2}. The positivity of this expression follows from the strict positivity of the function $\varphi$.
 %  \begin{align*} %\label{eqn:mink-gen2}
%  \frac 1\eta \int_0^\infty \eps^{D-d-1} \varphi(\eps)\ d\eps.
%\end{align*}
 This completes the proof.
\end{proof}

To derive a formula for the Minkowski content in terms of the generator, the projection condition comes into play, which will allow us to derive a nicer expression for $\varphi$.
The following observations are essential for the proof of Theorem~\ref{thm:Minkowski-measurability-fractals2}.

\begin{lemma}
   \label{lem:PC} Let $F$ be a self-similar set and $O$ a feasible open set for $F$.
    \begin{itemize}
      \item[(i)] If $O$ satisfies the projection condition \eqref{eqn:PC}, then, for each $\eps>0$ and $i=1,\ldots, N$,
\begin{align}
  \label{eqn:PC2} F_\eps\cap S_i O=(S_i F)_\eps\cap S_i O\,.
\end{align}
      \item[(ii)] If \eqref{eqn:PC} does not hold for $O$, then there exist some $j\in\{1,\ldots,N\}$ and some constants $c,\eps_1,\eps_2>0$ with $\eps_1<\eps_2\leq r_j\tilde g$ such that $\lambda_d(F_\eps\setminus (S_j F)_\eps \cap S_j O )\geq c$ for all $\eps\in[\eps_1,\eps_2)$.
    \end{itemize}
\end{lemma}
\begin{proof}
   (i) One of the set inclusions in \eqref{eqn:PC2} is obvious from $S_iF\subset F$. To see the reverse inclusion, let $x\in F_\eps\cap S_i O$.
      Then $d(x,F)\leq \eps$ and, by \eqref{eqn:PC}, $x\in\cl{\pi^{-1}_F(S_i F)}$.
   The latter means that we can find a sequence $(y_n)_n$ of points in $\pi^{-1}_F(S_i F)$ which converges to $x$ as $n\to\infty$. Since $y_n\in\pi^{-1}_F(S_i F)$ implies $d(y_n,S_i F)\leq d(y_n, F\setminus S_iF)$ for each $n\in\N$, the same must hold for the limit point $x$. Hence, since $S_iF$ is closed, there exists a point $z\in S_i F$ such that $d(x, S_iF)=|x-z|=d(x,F)\leq \eps$. But this implies $x\in(S_i F)_\eps\cap S_iO$, which completes the proof of (i).

   (ii) Assume that \eqref{eqn:PC} does not hold. Then there exists some $j\in\{1,\ldots, N\}$ and some $x\in S_j O$ such that $x\notin\cl{\pi_F^{-1}(S_j F)}$. Since this set is closed, we can find $\delta>0$ such that $B(x,\delta)\subseteq S_j O\cap \left(\cl{\pi_F^{-1}(S_j F)}\right)^c$. Let $d_1:=d(x,F)$ and $d_2:=d(x, S_j F)$. Since obviously $d_1<d_2\leq r_j\tilde g$, the numbers $\eps_1:=d_1+\frac{d_2-d_1}3$ and $\eps_2:=d_2-\frac{d_2-d_1}3$ satisfy $0<\eps_1<\eps_2<r_j\tilde g$. Now let $\tilde \delta:=\min\{\delta,(d_2-d_1)/3\}$. For any point $y\in B(x,\tilde\delta)$ and each $\eps\in[\eps_1,\eps_2)$, we have
   $$
   d(y,F)\leq d(x,F)+d(y,x)\leq d_1+\tilde\delta \leq d_1+\frac{d_2-d_1}3=\eps_1\leq \eps,
   $$
   and
   $$
   d(y,S_j F)\geq d(x,S_j F)- d(y,x)\geq d_2-\tilde\delta\geq d_2-\frac{d_2-d_1}3=\eps_2\geq \eps.
   $$
   This implies $B(x,\tilde \delta)\subset F_\eps\setminus (S_j F)_\eps$ for all $\eps\in[\eps_1,\eps_2)$. We conclude that, for all $\eps\in[\eps_1,\eps_2)$,
   $
   \lambda_d(F_\eps\setminus (S_j F)_\eps \cap S_j O ))\geq \lambda_d(B(x,\tilde \delta))=:c>0.
   $
   This completes the proof of (ii).
   \end{proof}

\begin{proof}[Proof of Theorem~\ref{thm:Minkowski-measurability-fractals2}]
   (i) The arguments for (i) are very similar to those in the proof of Theorem~\ref{thm:Minkowski-measurability-fractals}.
   The inclusion $\setmap O\subset O$ implies $O^c\subset(\setmap O)^c\subset((\setmap O)^c)_\eps$ for any $\eps>0$.
   %Moreover, the compatibility condition \eqref{eqn:compatibility} implies that $F_\eps\cap K_\eps\setminus K= K_\eps\setminus K$, cf.~\eqref{eqn:compat2}.
   Since $O$ is assumed to be strong, we can use Lemma~\ref{lem:key} %from \cite[Corollary~5.6.3]{winter08}. (Note that here the assumption is used that $O$ is a strong feasible open set!),
   and infer that there exist some constants $c,\gamma>0$ such that for all $\eps\in(0,1)$ the estimate
   \begin{align}\label{eqn:V-est2}
   \lambda_d(F_\eps\cap O^c)\leq c\eps^{d-D+\gamma}
   \end{align}
   holds. This implies immediately that $\sM^D(F,O^c)=0$ as claimed.

   Now observe that,
   %$\cl{O}=\cl{T}=T\cup\bd T$ and that $\lambda_d(\bd T)=0$???. Therefore, we have
   for all $\eps>0$,
    \begin{align*} %\label{eqn:compat3}
     \eps^{D-d}\lambda_d(F_\eps)=\eps^{D-d}\lambda_d(F_\eps\cap O)+\eps^{D-d}\lambda_d(F_{\eps}\cap O^c).
      \end{align*}
Taking the limit on both sides as $\eps \searrow 0$, the second term on the right does always tend to zero. We conclude that the limit $\sM^D(F)$ on the left hand side exists if and only if the limit $\sM^D(F,O)$ of the first term on the right hand side exists, and that both numbers coincide in this case.

(ii) We first show that the formula does not hold in the case $D=d$.
Indeed, any feasible set $O$ of such a self-similar set $F$ satisfies $\cl{O}=F$, cf.~\cite[Proposition 5.4 and Corollary 5.6]{GeometryOfSST}. Therefore, $\lambda_d(F_\eps\cap O)=\lambda_d(O)$ for any $\eps>0$ and so $\sM^d(F,O)=\lambda_d(O)>0$. (Thus, by part (i), $\sM^d(F)$ exists and is strictly positive, regardless whether the set is lattice or nonlattice!) On the other hand, the set $G=O\setminus\setmap\cl{O}$ is empty, since $O\subset F=\setmap F=\cl{\setmap O}$. (The tiling is not defined in this case.) The set $\Gamma=O\setminus\setmap O$ is not necessarily empty, but it is contained in $O$ and thus $F$ is dense in $\Gamma$. It follows that $\lambda_d(F_\eps\cap\Gamma)=\lambda_d(\Gamma)$ for any $\eps>0$ and thus the integral on the right hand side of \eqref{eqn:rel-mink-gen} is either zero (in case $\lambda_d(\Gamma)=0$) or $\infty$ (in case $\lambda_d(\Gamma)>0$). In any case the right hand side does not coincide with the positive and finite Minkowski content on the left.
Hence formula \eqref{eqn:rel-mink-gen} does not hold for full dimensional self-similar sets.

For the remainder of the proof, we can thus assume $D<d$. First we apply Proposition~\ref{thm:main2}, from which the existence of $\asM^D(F,O)$ (and of $\sM^D(F,O)$ in the nonlattice case) follows as claimed but with a different expression (given in \eqref{eqn:Yd-2}). By part (i), the existence of $\asM^D(F)$ (and in the nonlattice case of $\sM^D(F)$) follows with the same expression \eqref{eqn:Yd-2}.

It remains to verify that \eqref{eqn:rel-mink-gen} holds if and only if \eqref{eqn:PC} is satisfied.
%Assuming first that the projection condition \eqref{eqn:PC} holds, we infer from Lemma~\ref{lem:PC}(i) that
Since $(S_i F)_\eps\subset F_\eps$, the function $\varphi$ in \eqref{eqn:Yd-2} (which is given by \eqref{eqn:h-def-2}) can be rewritten as follows:
\begin{align}
   \varphi(\eps)&= \lambda_d(F_\eps\cap O)-\sum_{i=1}^N  \ind{(0,r_i{\tilde{g}}]}(\eps) \lambda_d(F_\eps\cap S_iO)+ \sum_{i=1}^N  \ind{(0,r_i{\tilde{g}}]}(\eps) \lambda_d(F_\eps\setminus(S_iF)_\eps\cap S_iO)\notag\\
   &=\lambda_d(F_\eps\cap O\setminus \setmap O)+\sum_{i=1}^N \ind{(r_i{\tilde{g}},\tilde g]}(\eps) \lambda_d(F_\eps\cap S_iO)+\sum_{i=1}^N \ind{(0,r_i{\tilde{g}}]}(\eps) \lambda_d(F_\eps\setminus(S_iF)_\eps\cap S_iO)\notag \\
   &=\lambda_d(F_\eps\cap \Gamma)+\lambda_d(O) \sum_{i=1}^N  r_i^d \ind{(r_i{\tilde{g}},\tilde g]}(\eps)+ \sum_{i=1}^N  \ind{(0,r_i{\tilde{g}}]}(\eps) \lambda_d(F_\eps\setminus(S_iF)_\eps\cap S_iO)\label{eqn:phi5}
   %&=\lambda_d(F_\eps\cap \Gamma)+\lambda_d(O) \sum_{i=1}^N  r_i^d \ind{(r_i{\tilde{g}},\tilde g]}(\eps),
\end{align}
for each $\eps\in(0,\tilde g]$. For the third equality, we have used that for $\eps>\tilde g$, $O\subset F_\eps$ and thus $S_iO\subset (S_iF)_\eps$, for $\eps>r_i\tilde g$. Therefore, $\lambda_d(F_\eps\cap S_iO)=\lambda_d(S_iO)=r_i^d \lambda_d(O)$, for each $\eps>r_i\tilde g$.

If we assume now, that the projection condition \eqref{eqn:PC} holds, then we can infer from Lemma~\ref{lem:PC}(i) that the last sum in the above representations of $\varphi$ vanishes.
Plugging the remaining representation of $\varphi$ into \eqref{eqn:Yd-2} and simplifying the integrals resulting from the second term, we get
\begin{align*}
  \eta \asM^D(F)&= \int_0^{\tilde{g}} \eps^{D-d-1} \lambda_d(F_\eps\cap \Gamma) d\eps + \lambda_d(O) \sum_{i=1}^N  r_i^d \int_{r_i\tilde g}^{\tilde{g}} \eps^{D-d-1} \ d\eps,\\
  &= \int_0^{\tilde{g}} \eps^{D-d-1} \lambda_d(F_\eps\cap \Gamma) d\eps + \lambda_d(O) \frac{\tilde g}{d-D}\left(1-\sum_{i=1}^N  r_i^d\right).
\end{align*}
Now observe that
%$O=\Gamma \cup \bigcup_{i=1}^N S_i O$ and this union is disjoint. Therefore,
$
\lambda_d(O)=\sum_{i=1}^N \lambda_d(S_iO)+\lambda_d(\Gamma)=\sum_{i=1}^N r_i^d \lambda_d(O)+\lambda_d(\Gamma)
$ and therefore,
\begin{align*}
   \lambda_d(\Gamma)=\left(1-\sum_{i=1}^N r_i^d\right) \lambda_d(O).
\end{align*}
Combining this with the observation that
\begin{align*}
   \int_{\tilde{g}}^\infty \eps^{D-d-1} \lambda_d(F_\eps\cap \Gamma) d\eps
   &=\int_{\tilde{g}}^\infty \eps^{D-d-1} \lambda_d(\Gamma) d\eps = \lambda_d(\Gamma) \frac{\tilde g^{D-d}}{d-D},
\end{align*}
we conclude that
\begin{align*}
  \eta \asM^D(F)&= \int_0^\infty \eps^{D-d-1} \lambda_d(F_\eps\cap \Gamma) d\eps,
\end{align*}
that is, formula \eqref{eqn:rel-mink-gen} holds.
%Up to now, we have not used the assumption $\lambda_d(\bd G)=0$, which implies $\lambda_d(\Gamma)=\lambda_d(G)$ since $\Gamma\setminus G\subset \bd G$. Employing it now, we get immediately formula \eqref{thm:main2}, proving that \eqref{eqn:PC} implies this formula.

For the reverse implication, we assume that \eqref{eqn:PC} does not hold and use Lemma~\ref{lem:PC}(ii), which implies that the last term in the above representation \eqref{eqn:phi5}
of $\varphi$ does not vanish for all $\eps$. There is some $j$ and an interval $[\eps_1,\eps_2)\subset(0,r_j\tilde g)$ on which the $j$-th term and thus the whole sum is bounded from below by some positive constant $c$. Plugging this into the formula  \eqref{eqn:Yd-2} for the Minkowski content, the above computations remain the same except that we have an extra term now, which is strictly positive:
\begin{align*}
  \eta \asM^D(F)&\geq \int_0^\infty \eps^{D-d-1} \lambda_d(F_\eps\cap \Gamma) d\eps + c\int_{\eps_1}^{\eps_2} \eps^{D-d-1} d\eps,
\end{align*}
Hence, formula \eqref{eqn:rel-mink-gen} does not hold in this case, which completes the proof of (ii).
   %for which we need to verify that $\dim_M(F, G)<D$. Indeed, the inclusion $G\subset (\setmap O)^c\subset ((\setmap O)^c)_\eps$ implies $\lambda_d(F_\eps\cap G)\leq \lambda_d(F_\eps\cap ((\setmap O)^c)_\eps)$ for any $\eps>0$. This allows to infer from \eqref{eqn:V-est}
%that, for any $\alpha>D-\gamma$, $\eps^{\alpha-d}\lambda_d(F_{\eps}\cap G)\to 0$ as $\eps \searrow 0$. Hence $\dim_M(F, G)\leq D-\gamma<D$ as claimed.

(iii) If $\lambda_d(\bd G)=0$, then $\lambda_d(G)=\lambda_d(\Gamma)$ and thus $\Gamma$ can be replaced by $G$ in \eqref{eqn:rel-mink-gen}.
\end{proof}

We emphasize again that Theorem~\ref{thm:Minkowski-measurability-fractals2} does not need any compatibility assumption. The derived formulas apply to self-similar sets in $\R^d$ of any dimension $D\in(0,d]$ and are in particular not restricted to dimensions $D>d-1$.

\begin{remark}(Recovering the compatible case.)\label{rem:compat-recover}
  In case the set $O$ generates a compatible tiling, we have $\lambda_d(F_\eps\cap O)=\lambda_d(F_\eps\cap T)=V(T,\eps)$ for each $\eps>0$ and thus $\asM^D(F,O)=\asM^D(\sT)$ (as well as $\sM^D(F,O)=\sM^D(\sT)$, whenever one of these limits exists). Therefore, Theorems~\ref{thm:Minkowski-measurability-fractals} and ~\ref{thm:Minkowski-measurability-fractals2} do both apply to this situation. Note that $D<d$ is necessary for the existence of $\sT(O)$ and that the projection condition is satisfied for $O$ if $\sT(O)$ is compatible. (Indeed, for $x\in S_iT$, there is some $\omega=\omega_1\ldots,\omega_m\in\Sigma_N^*$ such that $x\in S_\omega G$. Note that $\omega_1=i$. By compatibility, $\pi_F(x)\in \bd S_\omega G\subset S_\omega F\subset S_i F$, whenever the metric projection $\pi_F(x)$ is defined, and thus $x\in \pi_F^{-1}(S_i F)$. Otherwise, that is, if $x$ is in the exoskeleton of $F$, we still have $d(x,S_iF)\leq d(x,F)$ and thus $\cl{\pi_F^{-1}(S_i F)}$. This shows $S_iT\subset \cl{\pi_F^{-1}(S_i F)}$. Since the latter set is closed, we conclude $S_iO\subset \cl{S_iO}=\cl{S_iT}\subseteq\cl{\pi_F^{-1}(S_i F)}$. Hence the projection condition holds.)
  This means that the previous results for compatible tilings, in particular Theorem~\ref{thm:Minkowski-measurability-fractals}, can be recovered from Theorem~\ref{thm:Minkowski-measurability-fractals2}.

  However, the results for general tilings discussed above cannot be recovered from Theorem~\ref{thm:Minkowski-measurability-fractals2}. In general, it makes a difference whether the parallel volume $V(T,\eps)$ of the tiling or the parallel volume $\lambda_d(F_\eps\cap T)$ of $F$ restricted to the tiling is studied.  Compatible tilings are exactly those for which the two approaches yield the same.
\end{remark}

\begin{remark} \label{rem:monophase-no-comp}
With Theorem~\ref{thm:Minkowski-measurability-fractals2} at hand it should be possible to strengthen Corollary~\ref{cor:monophase} as follows: One can drop the assumptions of compatibility and of a monophase generator and assume instead that the parallel volume $\lambda_d(F_\eps\cap G)$ is a polynomial in $\eps\in(0,\tilde g)$. Then Theorem 4.8 in \cite{LPW2} does not apply directly but the methods used in the proof of this result can be adapted to the present setting. A simple example of this situation ist provided by the modified carpet in \cite[Example~2.4.5]{winter08}. The interior of its convex hull is a strong feasible set which is not compatible and the function $\lambda_d(F_\eps\cap G)$ of the  generator $G$ of the associated tiling is a polynomial.
%However, we are not aware of any self-similar set and a corresponding tiling such that $\lambda_d(F_\eps\cap G)$ is a polynomial. Hence we do not know yet, how useful this strengthening is and therefore we omit the details.
\end{remark}

\begin{remark}
   In Theorem~\ref{thm:Minkowski-measurability-fractals2} we only discussed the case of strong feasible sets $O$. It is clear from the discussion after Theorem~\ref{thm:Minkowski-measurability-fractals}, that this assumption cannot be omitted in general. However, similarly as Theorem~\ref{thm:Minkowski-measurability-fractals}, the statement of Theorem~\ref{thm:Minkowski-measurability-fractals2} remains true if the strong open set is replaced by a feasible set $O$ such that an estimate of the type \eqref{eqn:V-est} holds. In fact, the assumption can be weakened to the requirement $\dim_M(F,(\setmap O)^c)<D$.
\end{remark}

%The only assumption in Theorem~\ref{thm:main2} is that the Minkowski dimension of $F$ relative to the generator $G$ is less than $D$, which is for instance always satisfied if $O$ is a strong feasible set as in the situation of Theorem~\ref{thm:Minkowski-measurability-fractals2}.
%
%Only the full dimensional case $D=d$ has to be excluded. This is very natural, as in in this case the formula \eqref{eqn:mink-gen2} does not hold. Indeed, any feasible set $O$ of such a self-similar set $F$ satisfies $\cl{O}=F$. Therefore, $\lambda_d(F_\eps\cap O)=\lambda_d(O)$ for any $\eps>0$ and so $\sM^d(F,O)=\lambda_d(O)>0$. Thus $\sM^d(F,O)$ exists and is strictly positive, regardless whether the set is lattice or nonlattice. On the other hand, the set $G=O\setminus\setmap\cl{O}$ is empty (cf.~\cite[Proposition 5.4]{GeometryOfSST} - the tiling is not defined in this case) and thus the integral on the right hand side of \eqref{eqn:mink-gen2} is zero.

\begin{remark}
   We have shown in Theorem~\ref{thm:Minkowski-measurability-fractals2} that the generator formula \eqref{eqn:rel-mink-gen} does not hold for full dimensional self-similar sets. In this case, even the lattice-nonlattice dichotomy breaks down, as the Minkowski content $\sM^d(F)$ always exists. This clarifies that all full-dimensional sets have to be excluded from Lapidus's conjecture. (Note that the relevant part of Lapidus' original conjecture was for self-similar sets in $\R^d$ with $d-1<D<d$, see \cite[Conjecture~3]{Lap:Dundee}.) Since Minkowski contents are independent of the dimension of the ambient space, also sets which are full-dimensional with respect to their affine hull have to be excluded.
\end{remark}

\section{Fractal curvatures for self-similar tilings} \label{sec:curv}

In analogy with the results obtained above for Minkowski contents, we will now introduce and study fractal curvatures of self-similar tilings. Apart from being interesting in their own right, our main motivation is to understand their relation with fractal curvatures of self-similar sets.
%to understand whether the coefficients $\kappa_k(G)$ in the polynomial expansion \eqref{eqn:def-prelim-Steiner-like-formula} of the inner parallel volume allow a curvature interpretation. Moreover, as
%self-similar sets and their associated tilings are closely related, we expect close relations between their fractal curvatures. Indeed,
%
%It turns out that under appropriate assumptions the fractal curvatures of the sets and their tilings coincide (just as in the case of the Minkowski content above) and that the tilings are thus a useful tool to
%simplify the computation of fractal curvatures of self-similar sets. As in the case of the Minkowski content, we derive formulas for fractal curvatures involving solely the curvature data of the generator of the tiling. For compatible tilings, the general regularity and curvature bound conditions used in \cite{zaehle11,WZ} and discussed further in \cite{winter11} which ensure the existence of fractal curvatures for self-similar sets, turn out to be sufficient also for the validity of these new generator-type formulas.

%can be replaced by corresponding regularity and curvature bound conditions on the generator of the tiling (which are shown to imply the ones on the self-similar sets in the classic results).

We start by recalling the definition of fractal curvatures for compact sets and introduce the straightforward modification for tilings. For a closed set $K\subset\R^d$ and $x\in\R^d\setminus K$, let $\Sigma_K(x)$ be the set of points $a\in K$ such that $|x-a|=d(x,K)$. The point $x$ is called \emph{critical} for $K$, if $x\in\conv\Sigma_K(x)$ and a radius $\eps>0$ is called \emph{critical} for $K$ if there exists a critical point $x$ for $K$ with $d(x,K)=\eps$. Otherwise, the radius $\eps>0$ is called \emph{regular} for $K$ (or a \emph{regular value} of $K$). For sets $K\subset\bR^d$, $d\le 3$, Lebesgue almost all $\eps>0$ are regular values of $K$, see \cite{Fu85}. In higher dimensions this is not true in general. The importance of this regularity notion lies in the fact that, for regular values $\eps$ of $K$, the curvature measures of the set $K_\eps$ are well defined.
%Let $K\subset\R^d$ be a closed set with a compact boundary.
We write $\widetilde{A}$ for the {\it closure of the complement} of a set $A$. If a value $\eps$ is regular for $K$, then the set $\widetilde{K_\eps}\,(:=\widetilde{(K_\eps)})$ has positive reach, cf.~\cite{Fu85}, and the boundary of $K_\eps$ is a Lipschitz manifold of bounded curvature in the sense of \cite{RaZa1}.
 %(For general $d$, a sufficient condition for this property is that $\eps$ is a regular value of the distance function of $K$ in the sense of Morse theory, cf.~\cite{Fu85}.)
 Therefore, {\it Lipschitz-Killing curvature measures} are determined for $\widetilde{K_\eps}$ (in the sense of Federer \cite{Federer} as curvature measures for sets with positive reach) and thus for $K_\eps$ via normal reflection:
 \begin{equation}
C_k(K_\eps,\mydot):=(-1)^{d-1-k}C_k(\widetilde{K_\eps},\mydot)\, ,~~k=0,\ldots,d-1\,,
\end{equation}
cf.~\cite{RaZa1}. Here the surface area ($k=d-1$) is included, which is, in fact, equivalently given by $C_{d-1}(K_\eps,\mydot):=\frac 12 \Ha^{d-1}(\bd (K_\eps)\cap\mydot)$ which extends to all distances $\eps>0$ regardless of any regularity. While $C_{d-1}(K_\eps,\mydot)$ is always positive, the other curvature measures are signed in general.
 %and the volume measure $C_d(K_\eps,\mydot):=\lambda_d(K_\eps\cap\mydot)$ is added for completeness.
 For more details on singular curvature theory and some background see \cite{RaZa1,RaZa2} and the references therein. %\cite{winter08, zaehle11}.\\

Let $F\subset\bR^d$ be a compact set and let $s\geq 0$.
%(typically one has to choose $s=D$e its (upper) Minkowski dimension by $D$.
Assume that almost all $\eps>0$ are regular for $F$ (implying that curvature measures $C_0(F_\eps,\cdot),\ldots,C_{d-1}(F_\eps,\cdot)$ of $F_\eps$ are defined for almost all $\eps$). It is well known that there are no critical values $\eps>\sqrt{2}\diam F$. Therefore, this is an assumption about small $\eps$. Denote by $C_k(F_\eps)$ the total mass and by $C_k^{\var}(F_\eps)$ the mass of the total variation measure of the (signed) measure $C_k(F_\eps,\cdot)$.
If the essential limit
\begin{align} \label{def:fc}
   \sC_k^s(F):=\esslim{\eps \searrow 0} \eps^{s-k} C_k(F_\eps)
\end{align}
exists, then this number is called the \emph{($s$-dimensional) $k$-th fractal curvature} of $F$.
Moreover, the \emph{average ($s$-dimensional) $k$-th fractal curvature} is the limit
\begin{align} \label{def:avfc}
   \asC_k^s(F):=\lim_{\delta \searrow 0}\frac{1}{|\ln\delta|} \int_\delta^1
\eps^{s-k} C_k(F_\eps) \frac{d\eps}{\eps}\,.
\end{align}
%see \cite{WZ, PokWi} for more details.
For self-similar (and also more general) sets $F\subset\R^d$ with $\dim_M F=D<d$, typically one has to choose $s=D$ to obtain  nontrivial limits. The existence of the (average) fractal curvatures $\sC_k^D(F)$ and $\asC_k^D(F)$ has been established in the last years for different classes of self-similar sets under various assumptions, see \cite{winter08,zaehle11,WZ,RatZa12,BZ13} and \cite{KK12,Kom,bohl13} for related results for self-conformal sets.
%: for self-similar sets with polyconvex parallel sets in \cite{winter08}, for random and deterministic self-similar sets satisfying some curvature bound condition in \cite{zaehle11,WZ}, with a weaker integrability assumption in \cite{RatZa12}, for self-conformal sets \cite{KK12,kombrink, bohl} and for curvature direction measures of self-similar sets in \cite{BZ13}. Some of the mentioned papers establish also result about related weak limits of the curvature measures.
These results show a similar lattice-nonlattice dichotomy for fractal curvatures as the one observed for  the Minkowski content in Gatzouras' Theorem: In the non-lattice situation, $\sC_k^D(F)$ exists, while for lattice sets only the existence of $\asC_k^D(F)$ is established in general.

In analogy with relative Minkowski contents, cf.~\eqref{eqn:inner_Mink_def}, it is possible to restrict the curvature measures in \eqref{def:fc} and \eqref{def:avfc} to some set $\Omega\subset\R^d$ and define \emph{relative fractal curvatures} of $F$ relative to $\Omega$: Let $k\in\{0,\ldots,d-1\}$ and $s\geq 0$.
Whenever the limits exist, let
 \begin{align} \label{eqn:rel-fc}
   \sC_k^s(F,\Omega):=\esslim{\eps \searrow 0}\eps^{s-k} C_k \left(F_{\eps},\Omega \right),
 \end{align}
and denote by $\asC_k^s(F,\Omega)$ the corresponding average limit.

 For our purposes, in particular \emph{inner fractal curvatures} of a bounded open set $U\subset\R^d$ are relevant, by which we mean $\sC_k^s(\bd U, U)$, that is, the fractal curvatures of $\bd U$ relative to $U$.
 For self-similar tilings $\sT$, it is convenient to write
\begin{align} \label{def:fc-T}
\sC_k^D(\sT)&:=\sC_k^D(\bd T, T) \quad \text{ and } \quad \asC_k^D(\sT):= \asC_k^D(\bd T,T),
\end{align} %and $\dim_M \sT:=\dim_M T$.
 where $T$ denotes as before the union of the tiles of $\sT$.

 \medskip
 %\begin{remark}
   Observe that \emph{inner fractal curvatures} $\sC_k^s(\bd U, U)$
   %of a bounded open set $U\subset\R^d$
   are equivalently given in terms of inner parallel sets $U_{-\eps}$, cf.~\eqref{eq:inner-par-set}, of $U$, which allows some more convenient notation. We say $\eps>0$ is \emph{(inner) regular} for an open set $U$, if $\eps$ is regular for $U^c$. Then, for each regular value $\eps>0$ of $U$, we define the curvature measures of $U_{-\eps}$ in the natural way by
 $$
C_k(U_{-\eps},\cdot):=C_k((U^c)_\eps,\cdot), \qquad k=0,\ldots,d-1.
$$
Note that $\bd U_{-\eps}\cap U=\bd (U^c)_\eps$. The definition includes the case $\eps\geq\rho(U)$, where $\rho(U)$ denotes the inradius of $U$, for which $\bd (U_{-\eps})\cap U=\bd (U^c)_\eps=\emptyset$ and therefore $C_k(U_{-\eps},\cdot)=0$. Thus there is a natural range for $\eps$ for a bounded open set $U$, namely the interval $(0,\rho(U))$. If we now assume that almost all $\eps\in(0,\rho(U))$ are (inner) regular for $U$, then $\sC_k^s(\bd U, U)$ of $U$ is equivalently given by the limit $\esslim{}_{\eps \searrow 0}\eps^{s-k} C_k(U_{-\eps})$, for $k=0,\ldots,d-1$.
 %Note that for consistency, we consider only the boundary of $U_{-\eps}$ contained in the interior of $U$ and study the curvature on this boundary. Since $\bd U_{-\eps}\cap U=\bd (U^c)_\eps$, for $k=0,\ldots,d-1$, the $k$-th curvature measure of $U_{-\eps}$ can be defined by
%provided the measure on the right hand side exists, for which it is enough to require that $\eps>0$ is a regular value for $U^c$.
Here $C_k(U_{-\eps}):=C_k(U_{-\eps},\R^d)$ denotes the total mass of $C_k(U_{-\eps},\cdot)$. Similarly,  $C_k^{\var}(U_{-\eps})$ is the mass of the total variation measure $C_k^\var(U_{-\eps},\cdot)$ of  $C_k(U_{-\eps},\cdot)$.

We are now ready to formulate the first main result on the existence of (average) fractal curvatures for self-similar tilings in $\R^d$. Recall that a self-similar tiling $\sT=\sT(O)$ generated on a feasible set $O$ is only defined, if the underlying IFS has similarity dimension $D<d$ (non-triviality). Recall that $g=\rho(G)$ denotes the inradius of the generator $G$ of $\sT$.

\begin{theorem} \label{thm:main-curv}
Let $\sT=\sT(O)$ be a self-similar tiling generated on a feasible open set $O$ and let $k\in\{0,\ldots,d-1\}$. Assume that the generator $G$ of $\sT$ satisfies the following conditions:
\begin{enumerate}[(i)]
\item If $k\le d-2$, almost all $\eps\in(0,g)$ are (inner) regular values of $G$.
\item There are constants $c,\gamma>0$ such that, for almost all $0<\eps< g$,
\begin{equation}\label{eqn:C_k-G}
C_k^{\var}(G_{-\eps})\le c \eps^{k-D+\gamma}.
\end{equation}
\end{enumerate}

Then $\eps^{D-k}C_k^{\var}(T_{-\eps})$ is bounded for $\eps\in(0,g)$. Moreover, the average $k$-th fractal curvature $\asC_k^D(\sT)$  of $\sT$ exists and is given by the formula
\begin{align}\label{eqn:Yk}
\asC_k^D(\sT)&=\frac{1}{\eta} \int_0^g \eps^{D-k-1} C_k(G_{-\eps})\ d\eps,
\end{align}
where as before $\eta= \sum_{i=1}^N r_i^D |\ln r_i|$.

If $\sT$ is nonlattice, then the $k$-th fractal curvature $C_k^D(\sT)$ of $\sT$ exists and is given by formula \eqref{eqn:Yk}.
\end{theorem}

Note that the hypothesis is formulated completely in terms of the generator $G$ of the tiling and that also the formula provided for the (average) fractal curvatures is expressed in terms of the curvatures of (the parallel sets of) $G$. The formula \eqref{eqn:Yk} for the fractal curvatures of $\sT$ is in a sense even simpler than the one for the Minkowski content in \eqref{eqn:mink-gen} as the integration is over the finite interval $(0,g)$ only.
%nd the essential reason is that the curvatures $C_k(G_{-\eps})$ of the inner parallel sets are zero for $\eps\geq g$ (since $\bd(G_{-\eps})\cap G=\emptyset$ and curvature measures are concentrated on the boundary).

In Theorem~\ref{thm:main-curv}, we have tried to formulate minimal assumptions needed to apply the Renewal Theorem. The regularity assumption (i) on $G$ is needed to ensure that the total curvatures $C_k(G_{-\eps})$ (and thus $C_k(T_{-\eps})$) are well defined for sufficiently many $\eps>0$. This assumption is always satisfied if sets in dimension $d\leq 3$ are considered, cf.~\cite{Fu85}. It cannot be omitted in higher dimensions. In view of Example~\ref{ex:Koch}, it is clear that there exist counterexamples for which this assumption fails. (This is in contrast to the situation of fractal curvatures for self-similar sets, where no counterexamples are known and where the regularity condition is conjectured to be always satisfied, see \cite[p.1]{zaehle11}.)

The assumption \eqref{eqn:C_k-G} should be compared to the condition $\dim_M(\bd G, G)<D$ in Theorem~\ref{thm:main}, which is equivalently given by \eqref{eqn:VG}. In terms of scaling exponents (as defined e.g.\ in \cite{PokWi}), this condition may be reformulated as follows: the $k$-th scaling exponent $s_k(\bd G, G)$ of $\bd G$ \emph{relative to $G$} is strictly smaller than the similarity dimension $D$.  % of the system ().
To this condition, similar remarks apply as to condition $\dim_M(\bd G, G)<D$ in Theorem~\ref{thm:main}. In particular, the condition is close to optimal and cannot be omitted. $s_k(\bd G, G)>D$ would imply $s_k(\bd T, T)>D$ such that the ($D$-dimensional) fractal curvatures of $\sT$  would not exist. Similarly as in Example~\ref{ex:Koch}, it is easy to construct tilings the generators of which do not satisfy (ii). %\eqref{eqn:C_k-G}.

We will now prove Theorem~\ref{thm:main-curv}. Later we will demonstrate that the assumptions of Theorem~\ref{thm:main-curv} are satisfied under compatibility and the usual regularity and curvature bound assumptions used for analogous results for self-similar sets.
For the proof of Theorem~\ref{thm:main-curv}, we need the following convergence result for curvature measures which is a consequence of \cite[Theorem 5.2]{RatSchmSp}.
%For a set $A\subset\R^d$ and $x\in\R^d\setminus A$, let $\Sigma_A(x)$ be the set of points $a\in A$ such that $|x-a|=d(x,A)$. Recall that $x$ is called \emph{critical} for $A$, if $x\in\conv\Sigma_A(x)$ and a radius $r>0$ is called \emph{critical} for $A$ if there exists a critical point $x$ for $A$ with $d(x,A)=r$. The radius $r>0$ is called a \emph{regular value} of $A$ otherwise.
\begin{prop} \label{prop:contin}
   Let $K\subset\R^d$ be a set such that $\bd K$ is compact. Let $\eps>0$ be a regular value of $K$ and let $(\eps_n)$ be a sequence of positive numbers such that $\eps_n\to\eps$ as $n\to \infty$. Then there is $n_0\in\N$ such that $\eps_n$ is regular for $K$ for each $n\geq n_0$ and, for $k=0,\ldots, d-1$, the curvature measures $C_k(K_{\eps_n},\cdot)$ converge weakly to $C_k(K_\eps,\cdot)$ as $n\to\infty$.
\end{prop}
\begin{proof}
  Assume first that $K$ is compact.
  Let $\tilde\eps:=\inf\{\eps_n:n\in\N\}$. Observe that $0<\tilde\eps\leq\eps$. Let $r$ be some number such that $0<r<\tilde\eps$. Let $A:=K_{\eps-r}$ and $A^n:=K_{\eps_n-r}$ for each $n\in\N$. Then, $(A^n)$ is a sequence of compact sets converging to $A$ in the Hausdorff metric as $n\to \infty$. (Similary, the parallel sets $A^n_r=K_{\eps_n}$ converge to $A_r=K_\eps$ as $n\to\infty$.) Therefore, the claim follows from \cite[Theorem 5.2]{RatSchmSp}, provided that $r$ is a regular value of $A=K_{\eps-r}$.
  But the regularity of $r$ for $A$ is clear from the assumed regularity of $\eps$ for $K$, see Lemma~\ref{lem:reg} below.

  If $K$ is not compact, we intersect $K$ with a sufficiently large (closed) ball $B$ such that $\bd K$ (and thus $\bd K_t$ for each $t>0$) is contained in the interior of $B$, apply the first part of the proof to the compact set $K\cap B$ and use that curvature measures are locally determined.
\end{proof}

\begin{lemma} \label{lem:reg}
   Let $A\subset\R^d$ be a compact set and $r>0$.
   \begin{enumerate}
     \item[(i)] Let $x\in\R^d\setminus A_r$. Then $x$ is critical for $A$ if and only if $x$ is critical for $A_r$.
     \item[(ii)] Let $t>0$. Then $r+t$ is critical for $A$ if and only if $t$ is critical for $A_r$.
   \end{enumerate}
\end{lemma}
\begin{proof}
  (ii) follows directly from (i). For a proof of (i), we can assume without loss of generality that $x=0$. Let $t:=d(0,A_r)$ and define the homothety $h:\R^d\to\R^d$ by $z\mapsto\frac{t+r}t z$. It is easy to see that $d(0,A)=r+t$. Our first claim is that $y\in\Sigma_{A_r}(0)$ if and only if $h(y)\in\Sigma_{A}(0)$. Indeed, $y\in\Sigma_{A_r}(0)$ implies in particular $|y-0|=t$ and $y\in \bd (A_r)$. Because of the latter, there must exists a point $y'\in A$ such that $|y-y'|=r$. We necessarily have $y'=h(y)$, i.e., $y'$ is on the ray from $0$ through $y$ and $|y'|=r+t$. (Assume $y'\neq h(y)$. Let $z$ be the point on $[0,y']$ s.t. $|z-y'|=r$.
  Since $|y'-z|+|z-0|=|y'-0|<|y'-y|+|y-0|=r+t$, we get $|z-0|<t$. But this is a contradiction to the definition of $t$, since clearly $z\in A_r$.) Therefore, we have $h(y)\in A$ and $|h(y)|=r+t$, which means $h(y)\in\Sigma_{A}(0)$. This proves $h(\Sigma_{A_r}(0))\subset \Sigma_{A}(0)$. The argument for the reverse inclusion is even simpler. If $y\in\Sigma_{A}(0)$, then $y':=h^{-1}(y)$ satisfies obviously $|y-y'|=r$, meaning that $y'\in A_r$, and $|y'|=t$. This implies $y'\in\Sigma_{A_r}(0)$.

  Now we have $0\in \conv \Sigma_{A_r}(0)$ if and only if $0$ can be written as a convex combination $\sum_{i=1}^m \lambda_i y_i$ of points $y_i\in\Sigma_{A_r}(0)$. But then $0=h(0)=\sum_{i=1}^m \lambda_i h(y_i)$ is a convex combination of points in $\Sigma_{A}(0)$, that is $0\in\conv \Sigma_{A}(0)$, which completes the proof of (i).
\end{proof}
%In contrast, monophase or pluriphase generators do not allow any general simplifications in the formulas.

Let $S:\R^d\to\R^d$ be a similarity with ratio $r>0$, $A\subset\R^d$ closed and $x\in\R^d\setminus A$. Then $x$ is critical for $A$ if and only if the point $S(x)$ is critical for $S(A)$. Therefore, $\eps>0$ is a critical value of $A$ if and only if $r\eps$ is a critical value of $S(A)$. This has the following immediate implications for the relation between the (inner) critical values of the generator $G$ and the union set $T$ of a self-similar tiling $\sT$. (By \emph{inner critical values} of an open set $U$ we mean the critical values of $U^c$.)

\begin{lemma} \label{lem:sN}
  Let $\sT=\sT(O)$ be a self-similar tiling with generator $G$ and union set $T$. %$=\bigcup_{\sigma\in\Sigma^*} S_\sigma G$.
  Denote by $\sN_G$ the set of (inner) critical values of $G$.
  Then the set
  $$
  \sN:=\bigcup_{\sigma\in\Sigma_N^*} r_\sigma\sN_G
  $$
  is the set of critical values of $T$. It satisfies the relation $r_i\sN\subset\sN$ for all $i=1,\ldots,N$. If $\sN_G$ is a Lebesgue null set, then so is $\sN$.
\end{lemma}

We omit a proof, since it is very simple. In the proof of Theorem~\ref{thm:main-curv} below we will use in particular that if $\eps\notin\sN$ then $\eps/r_i\notin\sN$ for each $i=1,\ldots,N$. Observe that $\sN\subset(0,g)$ and that the curvature measures $C_k(G_{-\eps},\cdot)$ and $C_k(T_{-\eps},\cdot)$ are well defined for each $\eps\notin\sN$.

 \begin{proof}[Proof of Theorem~\ref{thm:main-curv}]
   We use \cite[Theorem~4.10]{RW09}, a modification of the Renewal Theorem 4.1.4 in \cite{winter08}, where the continuity assumption on $\varphi_k$ is weakened to continuity Lebesgue almost everywhere. Remark~\ref{rem:renewal} applies to this slightly more general statement equally as before.
   %Set $\tilde \sN':=\sN\cup\bigcup_{i=1}^N r_i^{-1}\sN\cap (0,g)$.
   Let the functions $f$ and $\varphi_k$ be defined by $f(\eps):=C_k(T_{-\eps})$ and $\varphi_k(\eps):=C_k(G_{-\eps})$ for $\eps\notin\sN$, and by $f(\eps)=\varphi_k(\eps):=0$ for $\eps\in\sN$ . Note that both functions are zero for $\eps\geq g$. (Therefore, we can omit the indicator functions $\ind{(0,r_ig]}$ in the formulas below.)
   Moreover, they satisfy the renewal equation
   \begin{equation} \label{eqn:curv-renewal}
\varphi_k(\eps)=f(\eps)-\sum_{i=1}^N r_i^{k} f(\eps/r_i),
\end{equation}
for all $\eps\notin\sN$ (that is, by (i) and Lemma~\ref{lem:sN}, for a.a. $\eps>0$). For $\eps\geq g$, this is obvious, since in this case both sides of the equation vanish. For $\eps\in(0,g)\setminus\sN$, this is seen from the relation
\begin{align*}
  C_k(T_{-\eps})=C_k(T_{-\eps},G) + \sum_{i=1}^N C_k(T_{-\eps}, S_iT)=C_k(G_{-\eps}) + \sum_{i=1}^N C_k((S_i T)_{-\eps}),
\end{align*}
which follows from the disjointness of the sets $S_iT$ and $G$, and the fact that curvature measures are locally determined (cf.\ e.g.~\cite[(1.5)]{WZ}). The observation that
\begin{align} \label{eqn:C-S_iT}
  C_k((S_i T)_{-\eps})=C_k(S_i (T_{-\eps/r_i}))=r_i^k C_k(T_{-\eps/r_i})=r_i^k  f(\eps/r_i),
\end{align}
 completes the proof of \eqref{eqn:curv-renewal}.

 The assumptions (i) and (ii) on the set $G$ imply that the hypothesis of \cite[Thm.~4.10]{RW09} is satisfied and the assertions of Theorem~\ref{thm:main-curv} follow directly from this theorem. In particular, Proposition~\ref{prop:contin} implies that $\varphi_k$ is continuous at each (inner) regular value $\eps\in(0,g)$ of $G$ and by (i) almost all $\eps$ are regular.
 %[22, Proposition 6]. (Recall that if $\eps>0$ is regular for $F$, then both sets $F_\eps$ and $\widetilde{F_\eps}$ are Lipschitz $d$-manifolds of bounded curvature in the sense of [22]. Therefore, by [22, Proposition 6], one has flat convergence of the normal cycles of $F_{\eps+\delta}$ to that of $F_\eps$ as $\delta\searrow 0$ for regular $\eps$ which implies vague convergence of the curvature measures and thus in particular convergence of the total curvatures.)
 Note that the formula \eqref{eqn:Yk} follows directly by plugging $\varphi_k$ into the general expression given in \cite[Thm.~4.10]{RW09}. No extra argument is needed here to derive the formula, in contrast to the derivation for the Minkowski content in Corollary~\ref{cor:main}. (Note that is enough to have the renewal equation satisfied for almost all $\eps$. One could easily redefine $\varphi_k$ on the null set $\sN$ such that \eqref{eqn:curv-renewal} holds for all $\eps$, and this would neither  affect the continuity of $\varphi_k$ almost everywhere nor the integral expression in the conclusion.)
 \end{proof}

\paragraph{\bf Relations between the fractal curvatures of self-similar sets and compatible tilings.}
Now we assume that the tiling $\sT$ satisfies the compatibility condition, that is, we assume $\bd G\subset F$. Recall that this condition is necessary and sufficient for the decomposition \eqref{eqn:compat2} to hold.
%where $K=\cl{O}$ and $T=\bigcup_{R\in\sT} R$ is the union set of the tiling.
The compatibility allows to relate the fractal curvatures of self-similar sets and associated tilings and to derive in this way generator formulas for the fractal curvatures of self-similar sets. Similar formulas have been obtained in \cite[Thm.~2.37]{Kom} under slightly stronger assumptions.  We will show that the assumptions on $G$ in Theorem~\ref{thm:main-curv} are implied by the usual regularity and curvature bound conditions on $F$ used e.g.\ in~\cite[Thm~2.1]{WZ}, see conditions (RC) and (CBC) below. Thus the generator formulas hold almost in the same generality as the previously known \emph{overlap} formulas (as e.g.\ in \cite[(3.6)]{WZ}) and do not need any extra assumptions apart from the existence of a strong feasible set $O$ which generates a compatible tiling.
%Beside the compatibility assumption it is important to work with strong feasible sets $O$.

\begin{theorem} \label{thm:comp}
Let $F$ be a self-similar set satisfying OSC and let $O$ be a strong feasible open set for $F$ such that the associated self-similar tiling $\sT=\sT(O)$ (with generator $G$) is compatible.
%the generator $G$ of satisfies the compatibility condition $\bd G\subset F$.
Let $k\in\{0,\ldots,d-1\}$.

If $k\le d-2$, assume additionally that $F$ satisfies the following conditions:
%for some $R>\sqrt{2}\, \diam(F)$:
\begin{enumerate}[(ii)]
\item[(RC)] Almost all $\eps>0$ are regular values for $F$.
\item[(CBC)] There are constants $c,\gamma>0$ and $R> \sqrt{2}\,\diam F$ such that, for almost all $0<\eps\le R$,
\begin{equation}\label{eqn:C_k-F}
C_k^{\var}(F_{\eps}, ((\setmap O)^c)_{\eps})\le c \eps^{k-D+\gamma}.
\end{equation}
%$C_k(G_{-\eps})$ (as a function in $\eps$) is continuous almost everywhere on $(0,g)$ and bounded on (This assumption can probably be removed!?) and that $s_k(G)<D$, %where $D$ is the Minkowski dimension of the associated self-similar set $F$.
\end{enumerate}

  Then, the average ($D$-dimensional) $k$-th fractal curvatures of $F$ and $\sT$ exist and coincide. Moreover, they are given by the formula %in \eqref{eqn:Yk}.
  \begin{align}\label{eqn:Yk2}
\asC^D_k(F)&=\asC^D_k(\sT)=\frac{1}{\eta} \int_0^g \eps^{D-k-1} C_k(G_{-\eps})\ d\eps.
\end{align}
%\begin{align}\label{eqn:Yk-G}
%\frac{1}{\eta} \int_0^g \eps^{D-d-1} C_k(G_{-\eps})\ d\eps,
%\end{align}
%where $\eta= - \sum_{i=1}^N r_i^D \ln r_i$.

If $\sT$ is nonlattice, then also the fractal curvatures $\sC^D_k(F)$ and $\sC^D_k(\sT)$ exist, coincide and are given by \eqref{eqn:Yk2}.
\end{theorem}

\begin{proof}
  For the assertions on the tiling $\sT$, we use Theorem~\ref{thm:main-curv}, for which we need to check that the assumptions (RC) and (CBC) imply the hypotheses (i) and (ii) of Theorem~\ref{thm:main-curv}. First, it is easy to see that the compatibility implies $\bd G_{-\eps}\cap G\subset \bd F_\eps$. So if $\eps\in(0,g)$ is a regular value for $F$, then it is also a regular value for $G^c$. Hence (RC) implies (i). For the second claim %that (CBC) implies (ii) of Theorem~\ref{thm:main-curv},
  observe that the compatibility assumption implies $\bd G_{-\eps}\cap G=\bd F_\eps\cap G$ from which we conclude, for regular values $\eps>0$ of $F$,
  $$
  C_k^{\var}(G_{-\eps})=C_k^{\var}(G_{-\eps},G)=C_k^{\var}(F_\eps, G),
  $$
  since the curvature measures are locally determined in the open set $G$, cf.~e.g.~\cite[property (1.5)]{WZ}. Now observe that $G\subset (\setmap O)^c\subset((\setmap O)^c)_\eps$. This implies the inequality
 $$
 C_k^{\var}(G_{-\eps})\leq C_k^{\var}(F_\eps, ((\setmap O)^c)_\eps),
 $$
 from which it is obvious that (CBC) implies condition (ii) of Theorem~\ref{thm:main-curv} as claimed. Note that $g<R$. In the case $k=d-1$, it follows from \cite[Lemma 4.8]{RW09} that the  condition (CBC) is always satisfied for strong open sets $O$ (and therefore we did not need to assume it). So also in this case (CBC) implies condition (ii). The assertions regarding the tiling $\sT$ follow thus directly from Theorem~\ref{thm:main-curv}.
 For the assertions regarding $F$, we employ the compatibility relation \eqref{eqn:compat2} to see that
 \begin{align} \label{eqn:compat3}
\eps^{D-k}C_k(F_\eps,F_\eps)=\eps^{D-k}C_k(F_\eps,T_{-\eps})+ \eps^{D-k}C_k(F_\eps, K_{\eps}\setminus K),
\end{align}
   for any regular value $\eps>0$ of $F$. Taking limits as $\eps\searrow 0$ in \eqref{eqn:compat3}, we first observe that
   \begin{align}\label{eqn:lim-K}
   \esslim{\eps \searrow 0} \eps^{D-k} C_k(F_\eps, K_\eps\setminus K)=0.
   \end{align}
 Indeed, (CBC) and the inclusions $K_\eps\setminus K\subset (\setmap O)^c\subset ((\setmap O)^c)_\eps$ yield that the essential limit of $\eps^{D-k} C_k^\var(F_\eps, K_\eps\setminus K)$ as $\eps\searrow 0$ is zero and then  \eqref{eqn:lim-K} follows from the inequality $|C_k(F_\eps, K_\eps\setminus K)|\leq C_k^{\var}(F_\eps, K_\eps\setminus K)$. In the case $k=d-1$, we use again \cite[Lemma 4.8]{RW09} for the same conclusion. %and the fact that absolute continuity implies continuity.

 The second observation is that the first term on the right hand side coincides with $\eps^{D-k}C_k(T_{-\eps})$, since $\bd F_\eps\cap T=\bd T_{-\eps}\cap T$. Thus, the essential limit as $\eps \searrow 0$ of this term is $\sC_k^D(\sT)$ and it exists if and only if the essential limit on the left hand side (that is, $\sC_k^D(F)$) exists. By the first part of the proof, $\sC_k^D(\sT)$ (and thus $\sC_k^D(F)$) exists in particular if $F$ is nonlattice. In general (in particular in the lattice case), we can argue similarly. Taking average limits on both sides of \eqref{eqn:compat3}, the second term on the right still vanishes, showing the equality of $\asC_k^D(F)$ and $\asC_k^D(\sT)$ whenever they exist. But $\asC_k^D(\sT)$ exists by the first part of the proof and is given by \eqref{eqn:Yk2}. This completes the proof.
\end{proof}

\begin{remark}
  %In \cite{zaehle11} and \cite{WZ}, the regularity is defined slightly different
  The condition (CBC) above is not exactly the condition used in \cite[Theorem 2.1]{WZ}, but (CBC) is implied by (and thus weaker than or at least equivalent to) the corresponding condition (ii) in \cite[Theorem 2.1]{WZ}, see \cite[Lemma 3.1]{WZ}. The proof of Theorem~\ref{thm:comp} shows that under the compatibility assumption, (CBC) implies condition (ii) of Theorem~\ref{thm:main-curv}. Hence, under compatibility, we lose no generality by working with the tilings instead of the sets. It is an interesting question, whether (under compatibility) (ii) is actually equivalent with (CBC) or whether it is strictly weaker.
\end{remark}

\begin{remark}
  In the proof of Theorem~\ref{thm:comp} we have used the assumptions of \cite[Theorem 2.1]{WZ} but not its conclusion. Thus Theorem~\ref{thm:comp} provides an independent proof of the existence of (average) fractal curvatures of self-similar sets, which is rather concise and simple, though restricted to self-similar sets which possess a compatible tiling generated on a strong feasible set.
  %The main achievement, however, is the alternative representation \eqref{eqn:Yk} of fractal curvatures of self-similar sets $F$ in terms of the generator. It greatly simplifies the computations of fractal curvatures.
\end{remark}

\paragraph{\bf Generator-type formulas for fractal curvatures without compatibility.} In view of Theorem \ref{thm:Minkowski-measurability-fractals2} for Minkowski contents, the question arises, whether one can also get rid of the compatibility assumption in Theorem~\ref{thm:comp}, and derive generator-type formulas for the fractal curvatures of self-similar sets in a more general setting. It is clear that without compatibility, the curvature measures $C_k(F_\eps,\cdot)$ and $C_k(G_{-\eps},\cdot)$ are not the same inside $G$, such that one has to look at the former now instead of the latter. More precisely, we will be interested in the fractal curvatures $\sC_k^D(F,O)$ of $F$ relative to a strong feasible set $O$, cf.~\eqref{eqn:rel-fc}. Again the projection condition \eqref{eqn:PC} plays an important role in absence of compatibility. Recall from Proposition~\ref{prop:Vc} that there is always a strong feasible set satisfying \eqref{eqn:PC}.  An additional problem now is that in general the intersections of $\bd F_\eps$ with the tile boundaries cannot be neglected in the case of curvature measures. To avoid this difficulty, we assume additionally that $C_k^\var(F_\eps,\bd O)=0$ for almost all $\eps>0$, which is a rather mild condition on the feasible set $O$ on which the tiling is generated. (Recall that we have the freedom to choose suitable sets $O$.) On the other hand, this condition ensures enough continuity of the relevant curvatures to apply again the Renewal Theorem. Moreover, it allows to write the formulas in terms of $G$ (rather than $\Gamma$). Note that under \eqref{eqn:PC}, the assumption implies that $C^\var_k(F_\eps,\bd G)=0$ for almost all $\eps>0$, since $\bd G\subset \bigcup_{i=1}^N \bd S_iO \cup \bd O$. Recall the definition of $\tilde g$ from \eqref{eq:rel-inr}.

%It is now essential to work with the set $\Gamma:=O\setminus \setmap O$ instead of the generator $G$. Note that in general $\Gamma$ is not an open set but we have $G\subseteq \Gamma\subseteq \cl{G}$ and thus difference $\Gamma\setminus G$ is contained in $\bd G$.
%Recall the definition of fractal curvatures relative to a set $\Omega$ from \eqref{eqn:rel-fc}.

\begin{theorem}[Generator formula for fractal curvatures of self-similar sets without compatibility]
  \label{thm:no-comp-curv}
Let $F$ be a self-similar set in $\bR^d$ satisfying OSC and with similarity dimension $D<d$.
Let $O$ be a strong feasible open set for $F$ satisfying the projection condition \eqref{eqn:PC} and  $C_k^\var(F_\eps,\bd O)=0$  for almost all $\eps>0$. Let $k\in\{0,\ldots,d-1\}$. If $k\le d-2$, assume additionally that $F$ satisfies the conditions (RC) and (CBC) of Theorem~\ref{thm:comp}.
%, and let $\sT=\sT(O)$ be the associated tiling with generator $G$.

Then $\sC_k^D(F,O^c)=0$. %(the $k$-th curvature of $F$ relative to $O^c$) is zero, i.e.\ $C_k^f(F,O^c)=0$.
Moreover, %the average $k$-th fractal curvature of $F$ coincides with the average $k$-th fractal curvature of $F$ relative to $O$, i.e.~
$\asC_k^D(F)$ and $\asC_k^D(F,O)$ exist and coincide. They are both given by the finite expression
\begin{align} \label{eqn:curv-no-comp-gen}
  \asC_k^D(F)&=\frac 1\eta \int_0^{\tilde g} \eps^{D-k-1} C_k(F_\eps,G)\ d\eps.
\end{align}
Furthermore, $\sC_k^D(F)$ exists if and only if $\sC_k^D(F,O)$ exists (and this happens in particular whenever $F$ is nonlattice). In this case, both quantities are given by the expression in \eqref{eqn:curv-no-comp-gen}.
\end{theorem}

The proof of this result %is similar to the proof of Theorem~\ref{thm:comp} and
is based on a suitable counterpart of Theorem~\ref{thm:main-curv} on the existence of $\sC_k^D(F,O)$.
%Note that the hypothesis in Proposition~\ref{thm:main-curv-O} below is again formulated completely in terms of the geometry inside the generator $G$ of $\sT(O)$, but this time in terms of the behaviour of $F_\eps$ in $G$.  By \emph{`$\eps$ is regular for $F$ in $G$'}, we simply mean that $F$ has no critical points in the set $\bd F_\eps\cap G$.

\begin{prop} \label{thm:main-curv-O}
Let $F$ be a self-similar set in $\bR^d$ satisfying OSC and with similarity dimension $D<d$.
Let $O$ be a strong feasible open set for $F$ satisfying the projection condition \eqref{eqn:PC} and  $C_k^\var(F_\eps,\bd O)=0$ for almost all $\eps\in(0,\tilde g)$. Let $k\in\{0,\ldots,d-1\}$. If $k\leq d-2$, assume additionally that the following conditions are satisfied:
\begin{enumerate}[(i)]
\item Almost all $\eps\in(0,\tilde g)$ are regular for $F$.
\item There are constants $c,\gamma>0$ such that for almost all $0<\eps< \tilde g$
\begin{equation}\label{eqn:C_k-G-O}
C_k^{\var}(F_{\eps},G)\le c \eps^{k-D+\gamma}.
 \end{equation}
%$C_k(G_{-\eps})$ (as a function in $\eps$) is continuous almost everywhere on $(0,g)$ and bounded on (This assumption can probably be removed!?) and that $s_k(G)<D$, %where $D$ is the Minkowski dimension of the associated self-similar set $F$.
\end{enumerate}
Then $\eps^{D-k}C_k^{\var}(F_\eps, O)$ is bounded for $\eps\in(0,g)$. Moreover, the average $k$-th fractal curvature $\asC_k^D(F,O)$ of $F$ relative to $O$ exists and coincides with the number
\begin{align}\label{eqn:Yk-Gamma}
\frac{1}{\eta} \int_0^g \eps^{D-k-1} C_k(F_\eps,G)\ d\eps.
\end{align}
%where as before $\eta= \sum_{i=1}^N r_i^D |\ln r_i|$.
If $F$ is nonlattice, then also the $k$-th fractal curvature $\sC_k^D(F,O)$ of $F$ relative to $O$ exists and equals the number in \eqref{eqn:Yk-Gamma}.
\end{prop}

 \begin{proof}[Proof of Proposition~\ref{thm:main-curv-O}]
%Denote by $\sN$ and $\sN^*$, respectively, the set of critical values $\eps>0$ of $F$ in $G$ and in $O$. Our first observation is that $\sN^*\subset\bigcup_{\sigma\in\Sigma^*} r_\sigma \sN$. To see this, let $r>0$ be a critical value of $F$ in $O$. Then there is some $\sigma\in\Sigma^*$ such that $F$ has a critical point in $S_\sigma \overline{G}\cap bd F_r$, since the sets $S_\sigma \cl{G}$ form a cover of $O\setminus F$ (and critical points of $F$ are by definition outside $F$). Indeed, by the properties of the tiling $\sT(O)$, we have $\cl{O}=\cl{\bigcup_{\sigma} S_\sigma G}$. Hence for any $x\in O$, either $x\in S_\sigma \cl{G}$ for some $\sigma\in\Sigma^*$ or there is a sequence $(\sigma_n)\in\Sigma^*$ such that $d(x,S_{\sigma_n}G)\to 0$ as $n\to\infty$. But the latter implies $x\in F$.  If $x$ is a critical point of $F$ in $S_\sigma \overline{G}\cap bd F_r$, then $r<r_\sigma \tilde g$ and $S^{-1}_\sigma x\in \cl{G}$ is a critical point of $F$ in $\overline{G}\cap bd F_{r/r_\sigma}$. This implies that $r/r_\sigma$ is a critical value for $F$ in $\cl{G}$ and thus $r\in r_\sigma \sN$.
%This proves the above assertion from which we can now conclude immediately, that the set $\sN^*$ of critical values of $F$ in $O$ is a Lebesgue null set, since $\sN$ has this property by assumption (i).
Let $\sN'$ be the set of critical values of $F$ and set $\sN:=\bigcup_{i=1}^N r_i \sN$. Note that $\sN\subset(0,\tilde g)$ and $\sN$ is a null set by condition (i).
% Since, by (i), $\sN$ is a null set, $\sN^*$ is a null set, too.
%It is convenient to redefine the curvature measures on $\sN^*$ by setting $C_k(F_\eps,\cdot):=0$ for $\eps\in\sN^*$.
Let $f(\eps):=C_k(F_{\eps},O)$ and $\varphi_k(\eps):=C_k(F_\eps,G)$, for $\eps\notin\sN$ and extend both functions to $\sN$ by setting $f(\eps)=\varphi_k(\eps)=0$.
   %where as before $\Gamma=O\setminus\setmap O$.
   The proof follows now essentially the lines of the proof of Theorem~\ref{thm:main-curv}. First we show that the new $f$ and $\varphi_k$ satisfy the renewal equation \eqref{eqn:curv-renewal} for all $\eps\notin\sN$. By definition of $\tilde g$, both functions are zero for $\eps\geq \tilde g$, such that \eqref{eqn:curv-renewal} holds trivially for $\eps>\tilde g$.

   Let $\eps\in(0,\tilde g)\setminus\sN$. Then $\eps$ is regular for $F$. Moreover, $\eps/r_i$ is regular for $F$ for each $i=1,\ldots,N$. (Assume not. Then $\eps/r_i\in\sN'$ and thus $\eps\in r_i\sN'\subset\sN$, a contradiction.)
   % where as before $\tilde g=\sup\{d(x,F):x\in G\}$.
   Let as before $\Gamma=O\setminus \setmap O$. Then the decomposition $O=\bigcup_{i=1}^N S_i O\cup \Gamma$ is disjoint and we have
   \begin{align}
  C_k(F_\eps, O)=\sum_{i=1}^N C_k(F_\eps, S_i O)+C_k(F_{\eps},\Gamma).
  %=\sum_{i=1}^N C_k(F_\eps(S_i T)_{-\eps})+C_k(G_{-\eps}),
\end{align}
 Since $\Gamma\setminus G\subseteq\bd G$ and $C_k^\var(F_\eps,\bd G)=0$, we can replace $\Gamma$ by $G$ in the above equation such that the last term on the right is $\varphi_k(\eps)$ (for a.a. such $\eps$). Furthermore, since the projection condition \eqref{eqn:PC} is assumed to hold, by Lemma~\ref{lem:PC} (i), we can replace $C_k(F_\eps, S_i O)$ by $C_k((S_iF)_\eps, S_i O)=r_i^k f(\eps/r_i)$ %(curvature measures are locally determined) %(cf.\ e.g.~\cite[(1.5)]{WZ}))
 in the above equation, from which  \eqref{eqn:curv-renewal} is transparent.

 The assumptions (i) and (ii) imply that the hypothesis of \cite[Thm.~4.10]{RW09} is satisfied and the assertions of Proposition~\ref{thm:main-curv-O} follow directly from this theorem. In particular, by Proposition~\ref{prop:contin}, the measures $C_k(F_\eps,\cdot)$ are weakly continuous in $\eps$ at every regular value $\eps$ of $F$. Therefore, $\varphi_k$ is continuous at those $\eps$, since, by the assumption $C_k^\var(F_\eps,\bd O)=0$, we have $C_k^\var(F_\eps,\bd G)=0$ and thus $G$ is a continuity set of the measure $C_k(F_\eps,\cdot)$.
\end{proof}

\begin{proof}[Proof of Theorem~\ref{thm:no-comp-curv}]
  We want to apply Proposition~\ref{thm:main-curv-O}, for which we need to check that (for $k\leq d-2$) the assumptions (i) and (ii) of this statement are satisfied. Obviously, (RC) implies (i).
 % For the assertions on the tiling $\sT$, we use Theorem~\ref{thm:main-curv}, for which we need to check that the assumptions (RC) and (CBC) imply the hypotheses (i) and (ii) of Theorem~\ref{thm:main-curv}. First, it is easy to see that the compatibility implies $\bd G_{-\eps}\cap G\subset \bd F_\eps$. So if $\eps\in(0,g)$ is a regular value for $F$, then it is also a regular value for $G^c$. Hence (RC) implies (i). For the second claim %that (CBC) implies (ii) of Theorem~\ref{thm:main-curv},
  Furthermore, (CBC) implies (ii) simply because $G\subset (\setmap O)^c\subset ((\setmap O)^c)_\eps$. %((CBC) can be equivalently reformulated with some $R$ such that $\tilde g\leq R$, cf.~\cite{Winter11}).
  The assertions regarding the relative fractal curvatures $\sC_k^D(F,O)$ and $\asC_k^D(F,O)$ follow thus directly from Proposition~\ref{thm:main-curv-O}.
 For the assertions regarding the fractal curvatures of $F$, we use the obvious equation
 \begin{align} \label{eqn:decomp}
\eps^{D-k}C_k(F_\eps)=\eps^{D-k}C_k(F_\eps,O)+ \eps^{D-k}C_k(F_\eps, O^c),
\end{align}
   which holds for any $\eps\notin\sN'$. Taking essential limits as $\eps\searrow 0$ in this equation, we first observe that
   \begin{align}\label{eqn:lim-Oc}
   \sC^D_k(F,O^c)=\esslim{\eps \searrow 0} \eps^{D-k} C_k(F_\eps,O^c)=0.
   \end{align}
 Indeed, (CBC) and the set inclusion $O^c\subset (\setmap O)^c \subset ((\setmap O)^c)_\eps$ yield that the essential limit of $\eps^{D-k} C_k^\var(F_\eps, O^c)$ as $\eps \searrow 0$ is zero. Thus, the claim \eqref{eqn:lim-K} follows from the inequality $|C_k(F_\eps, O^c)|\leq C_k^{\var}(F_\eps, O^c)$. %and the fact that absolute continuity implies continuity.
 Now the remaining assertions of Theorem~\ref{thm:no-comp-curv} follow from \eqref{eqn:decomp} and \eqref{eqn:lim-Oc} similarly  as in the proof of Theorem~\ref{thm:comp}.
 %
%
% Therefore, the left hand side of \eqref{eqn:decomp} converges (to $\sC^D_k(F)$) if and only if the first term on the right hand side of converges to $\sC^D_k(F,O)$, and the latter exists by the first part of the proof, provided $F$ is nonlattice. In general (in particular, in the lattice case) we can argue similarly. Taking average limits on both sides of \eqref{eqn:decomp}, the second term on the right still vanishes, showing the equality of the average fractal curvatures of $F$ and $\sT$.
% But $\widetilde{C}_k^f(\sT)$ exists by the first part of the proof. Thus $\widetilde{C}_k^f(F)=\widetilde{C}_k^f(\sT)$ and both fractal curvatures are given by the the formula \eqref{eqn:Yk} as claimed.
\end{proof}

\begin{remark}
   It is worth noting that the validity of the conditions (RC) and (CBC) in Theorems~\ref{thm:comp} and \ref{thm:no-comp-curv} does not depend on the choice of the feasible set $O$. For (RC) this is obvious and for (CBC) this follows from \cite[Corollary 4.9]{Winter11}. It is possible to reformulate (CBC) in such a way that the set $O$ is not used. Hence these conditions do not impose any additional restrictions on the choice of $O$.
   On the other hand, the compatibility in Theorem~\ref{thm:comp} and the projection condition \eqref{eqn:PC} together with the \emph{boundary condition} $C_k(F_\eps,\bd O)=0$ for a.a. $\eps>0$ in Theorem~\ref{thm:no-comp-curv} clearly depend on the choice of $O$. While compatibility is definitely not satisfiable for all self-similar sets, cf.~Remark~\ref{rem:compatible-limit} above, Proposition~\ref{prop:Vc} shows that every self-similar set (satisfying OSC and $D<d$) possesses a strong feasible set $O$ such that \eqref{eqn:PC} is satisfied. Hence the projection condition does not lessen the class of sets covered by Theorem~\ref{thm:no-comp-curv}. Only the boundary condition may impose some additional restriction on this class, since it may happen that no set $O$ satisfies both \eqref{eqn:PC} and this condition.
\end{remark}

\begin{remark}
   The results imply that the fractal curvatures $\sC^D_k(F)$ are finite but (for $k\leq d-2$) they are not necessarily positive. They can assume negative values and they can also be zero. There exist non-trivial self-similar sets for which the similarity dimension $D$ is not the right scaling exponent for the $k$-th curvature measure. Such sets are studied in detail in \cite{PokWi}.
\end{remark}

\begin{remark}
   In Theorem~\ref{thm:Minkowski-measurability-fractals2} we have given a complete characterization of the existence of generator-type formulas for the Minkowski content of a self-similar set based on a strong feasible set $O$.
   The corresponding statement for fractal curvatures, Theorem~\ref{thm:no-comp-curv}, is not quite as strong. The existence of a strong feasible set $O$ satisfying the projection condition and the boundary condition $C_k(F_\eps,\bd O)=0$ for a.a $\eps$, ensures the existence of a generator type formula. But the converse is probably not true in general. Since curvature measures are signed for $k\leq d-2$, it could happen that a generator-type formula holds even if the projection condition fails, because different contributions from the extra terms may cancel each other.
\end{remark}

\begin{remark}
  The regularity condition (i) in Proposition~\ref{thm:main-curv-O} can be weakened as follows: for almost all $\eps\in(0,\tilde g)$, there is no critical point of $F$ in the set $\bd F_\eps\cap \cl{G}$.
  Indeed, it follows from the observation that the assertion of Lemma~\ref{lem:sN} holds equally with $G$ replaced by $\cl{G}$ and from the fact that $\{S_\sigma \cl{G}:\sigma\in\Sigma_N^*\}$ is a cover of $\cl{O}\setminus F$. Note that critical points of $F$ are by definition outside $F$. Hence the above condition implies that for a.a. $\eps$ there are no critical points of $F$ in the set $\bd F_\eps\cap O$, which is all that is needed in the proof of Proposition~\ref{thm:main-curv-O}.
\end{remark}

\paragraph{\bf The case $k=d-1$.} The measure $C_{d-1}(F_\eps,\cdot)$ coincides with $\frac 12 \Ha^{d-1}(\bd F_\eps\cap\cdot)$ and the assumptions (RC) and (CBC) in Theorem~\ref{thm:no-comp-curv} are not needed, since the surface area is well defined for any parallel set $F_\eps$ and a condition analogous to (CBC) is always satisfied, cf.~\cite[Lemma 4.8]{RW09}. The corresponding (average) fractal curvature is (up to a normalisation constant) the (average) S-content, discussed in \cite{RW09}. The S-content was shown in \cite{RW12} to coincide (for arbitrary bounded sets) with the Minkowski content in case one of these contents exists as a positive and finite number. This allows to derive an alternative formula for the Minkowski content from Theorem~\ref{thm:comp} or, more generally, from Theorem~\ref{thm:no-comp-curv}:

\begin{cor} \label{cor:comp}
Let $F$ be a self-similar set in $\bR^d$ satisfying OSC and with $D<d$.
Let $O$ be a strong feasible open set for $F$ satisfying the projection condition \eqref{eqn:PC}. Assume $\Ha^{d-1}(F_\eps\cap\bd O)=0$ for almost all $\eps\in(0,\tilde g)$.

 Then, the average Minkowski content $\asM^D(F)$ is given by the alternative expression \begin{align}\label{eqn:Yd-G2}
\frac 1{d-D}\frac{1}{\eta} \int_0^{\tilde g} \eps^{D-d} \Ha^{d-1}(\bd F_\eps\cap G) \ d\eps=\frac 1{d-D}\frac{1}{\eta} \int_0^{\tilde g} \eps^{D-d} \frac d{d\eps}\lambda_d(F_\eps\cap G)\ d\eps\,.
\end{align}
%where $\eta= - \sum_{i=1}^N r_i^D \ln r_i$.
If $F$ is nonlattice, then the Minkowski content $\sM^D(F)$ is given by the same expression.
\end{cor}
However, the derived expression does not provide much additional insight (as we were hoping for). Using the differentiability properties of the volume function, the new formula can, in fact, be obtained directly from \eqref{eqn:rel-mink-gen} using integration by parts. Note that the assumption $\Ha^{d-1}(F_\eps\cap\bd O)=0$ for a.a.~$\eps$ is equivalent to $\lambda_d(\bd G)=0$ in part (iii) of Theorem~\ref{thm:Minkowski-measurability-fractals2}.

\begin{rem}
  All results about self-similar tilings, in particular Theorems~\ref{thm:main} and \ref{thm:main-curv}, apply equally to self-similar fractal sprays in $\R^d$, as discussed e.g.~in \cite{FGCD}.
\end{rem}

\paragraph{\bf Acknowledgements.} During the work on this article the author was supported by DFG project no.\ WI 3264/2-2.  I am very grateful to Erin Pearse for valuable comments on an earlier version, in particular for his observation that the central open set satisfies the projection condition, and for providing Figure~\ref{fig:KC}.

\bibliographystyle{plain}%{math}
\bibliography{strings}

\end{document}